\documentclass[a4paper,preprint,showkeys]{amsproc}
\usepackage{mathrsfs}
\usepackage{amsmath,booktabs}
\usepackage[T4, OT1]{fontenc}
\usepackage[arrow, matrix, curve]{xy}
\usepackage{newunicodechar}
\usepackage{graphicx}
\usepackage{bm}
\usepackage[dvips]{epsfig}
\usepackage{rotating}
\usepackage{amsmath,amssymb,amsfonts,amsthm}
\usepackage{amsmath,amscd}
\usepackage{subfig}
\usepackage{caption}
\usepackage{pst-all}
\usepackage{dcolumn}
\usepackage{amsthm}
\usepackage{dsfont}
\usepackage{amssymb}
\usepackage[hyphens]{url} 
\usepackage[dvips]{graphicx}
\usepackage{mathrsfs}
\makeatletter
\@namedef{subjclassname@2020}{\textup{2020} Mathematics Subject Classification}
\makeatother
\theoremstyle{plain}
 \newtheorem{thm}{Theorem}[section]
 
 \newtheorem{lem}{Lemma}[section]
 \newtheorem{cor}{Corollary}[section]
\theoremstyle{definition}
 \newtheorem{exm}{Example}[section]
 \newtheorem{dfn}{Definition}[section]
\theoremstyle{remark}
 \newtheorem{rem}{Remark}[section]
 \numberwithin{equation}{section}

\renewcommand{\leq}{\leqslant}\renewcommand{\geq}{\geqslant}
\newcommand{\eg}{e.\,g.\ }
\newcommand{\ie}{i.\,e.\ }
\renewcommand{\setminus}{\smallsetminus}
\title[Fractional derivative of tempered distributions \\ A new approach]{Fractional calculus of tempered distributions.\\ A new approach.}

\author[Belardinelli]{Cyril Belardinelli} 
\address{ 
Kantonsschule Solothurn\\
Solothurn\\
Switzerland}
\email{cyril.belardinelli@ksso.ch}
\begin{document}
\vspace{18mm} \setcounter{page}{1} \thispagestyle{empty}
\begin{abstract}
In the present article, a new method for the evaluation of fractional derivatives of arbitrary real order is proposed.
Numerous but inequivalent formulations have been given in the past. Some of them exhibit unsatisfactory properties such as \eg the absence of a semigroup property.
It is shown that the proposed definition is free from such drawbacks. In order to achieve such a generalization one has to work in the context of tempered distributions (generalized functions) where the concept works nicely. The last part of the article shows the possibility of defining fractional derivatives of periodic functions (Fourier series). 
\end{abstract}
\maketitle
\section{\label{intro}Introduction}
An essential part of mathematical work consists in extending the range of validity of given concepts to more general situations.  
Well known examples are Sobolev's generalized (weak) derivative
and the extension of factorials of integers to factorials of complex numbers by introducing the Gamma function. In the past many attempts have been made to introduce derivatives of non-integer order, commonly known as fractional derivatives.\\
Some well known definitions are due to J. Liouville, G. F. B.
Riemann, A. Gr\"unwald, A. Letnikov, M. Riesz, A. Erd\'elyi, H. Kober, and, more recently, to M. Caputo.
Unfortunately, these definitions give rise to different results when computing the derivatives of the
same function. 
Moreover, it turns out that some of the formulations give rise to various inconsistencies and/or unwanted features. This occurs \eg for the Riemann-Liouville fractional derivative where the fractional derivative of a constant function is non-zero. One can raise doubts if the concept of a \emph{function's} fractional derivative is meaningful at all, at least if one wants to preserve the properties i)\,-\,iv) given below.\\
An interesting question arises when one asks about fractional derivatives of \emph{tempered distributions} (generalized functions) rather than \emph{functions}. Motivated by that question we seek for a generalization by using the theory of generalized functions. \cite{schwartz:1950, schwartz:1951},\cite{gelfand:1964}.\\
By using the well-known fact that every tempered distribution (subsequently denoted by $T[\varphi]$) has a well defined derivative and is Fourier transformable as well it is possible to obtain a consistent definition for the fractional derivative of arbitrary real order. We stress the fact that the proposed definition as given by Eq.~\ref{the_definition} in Sec.~\ref{sec_frac_deriv} covers also fractional anti-derivatives corresponding to derivatives of negative fractional order.\\
The main purpose of the present paper is to show that such an extension is possible in a coherent way provided that one operates within the space of tempered distributions $\mathscr{S^{\prime}(\mathbb{R})}$. These results are exposed in Sec.~\ref{sec_frac_deriv}. The content of Sec.~\ref{Sec.1} 
 is devoted to the evaluation of several distributional Fourier transforms one needs for the definition of fractional derivatives. In this sense, Sec.~\ref{Sec.1} can be viewed as a preparative. However, Sec.~\ref{Sec.1} section has its own raison d'\^etre since some new formulae are derived such as \eg the Fourier transform of the distribution $\mathrm{Pf}\,x^{-\alpha}$ ($\alpha \in \mathbb{R}$) and a general formula for the Fourier transform of rational functions.\\
This paper is written in a didactic spirit. Exposed derivations are perhaps more detailled than necessary. Numerous examples are given in every section. This article deals only with functions (generalized or not) of one variable. 
\\To be explicit, in Sec.~\ref{sec_frac_deriv} we introduce the set of linear operators $\mathscr{D}^{\alpha}\colon\mathscr{S^{\prime}(\mathbb{R})}\to \mathscr{S^{\prime}(\mathbb{R})}$ depending on a real parameter $\alpha$ satisfying the following properties:\\
\begin{itemize}
\label{properties}
\item[i)] \emph{Extension property}: $\mathscr{D}^{\alpha}T[\varphi]=(-1)^{\alpha}\,T[\varphi^{(\alpha)}]$ for every positive integer $\alpha$.\\ 
\item[ii)] Each of the two sets  $(\mathscr{D}^{\alpha})_{\alpha\geq 0}$ and $(\mathscr{D}^{\alpha})_{\alpha\leq 0}$ forms a commutative ring under canonical addition and multiplication. The multiplicative identity is given by $\mathds{1}:=\mathscr{D}^{0}$.\\
\item[iii)] The semigroup property $\mathscr{D}^{\alpha}\mathscr{D}^{\beta}\,T[\varphi]=\mathscr{D}^{\alpha+\beta}\,T[\varphi]$ holds separately in each of the two sets $(\mathscr{D}^{\alpha})_{\alpha\geq 0}$ and $(\mathscr{D}^{\alpha})_{\alpha\leq 0}$.\\
\item[iv)] Fractional derivatives $\mathscr{D}^{\alpha}$ of constant distributions (\ie generated by constant functions) vanish for every positive fractional order $\alpha$.
\end{itemize}
\begin{rem}
property ii) means that $(\mathscr{D}^{\alpha},+)_{\alpha\geq 0}$ is an abelian group and $(\mathscr{D}^{\alpha},\cdot)_{\alpha\geq 0}$ is a semi-group with multiplicative identity defined in ii). The same is true for the set (of fractional anti-derivatives): $(\mathscr{D}^{\alpha})_{\alpha\leq 0}$. 
\end{rem}
\section{Preliminaries}
\subsection{Hadamard Regularization of divergent integrals}
The function\\ $f(x)=x^{-\alpha}$ 
has a non-integrable singularity in $x=0$ in the case $\alpha \geq 1$. However, for the purpose of Sec.~\ref{sec_frac_deriv} we need to turn the function $f(x)=x^{-\alpha}$ into a tempered distribution for a general real value $\alpha$. This goal can be achieved by using \emph{Hadamard's finite part} regularization $\mathrm{Pf}$ (Partie finie)
which is known to generalize \emph{Cauchy's Principal value} $\mathrm{Vp}$ (Valeur principale).\\
\\
The formal definition is given by (with an arbitrary $\alpha \in \mathbb{R}$):
\begin{equation}
\label{specific_distribution}
\mathrm{Pf}\frac{1}{x^{\alpha}}:=\mathrm{Pf}\int\limits_{\mathbb{R}}\frac{\varphi(x)}{x^{\alpha}}\mathrm{d}x=\mathrm{f.\,p.\,of}\lim_{\epsilon\to0^{+}}\int\limits_{|x|>\epsilon}\frac{\varphi(x)}{x^{\alpha}}\mathrm{d}x
\end{equation}
\begin{exm}
Well-known examples are given by the following distributions \cite{vladimirov:2002}:
\begin{equation*}
\begin{split}
\mathrm{Vp}\frac{1}{x}:&=\lim_{\epsilon\to0^{+}}\int\limits_{|x|>\epsilon}\frac{\varphi(x)}{x}\mathrm{d}x\\
\mathrm{Pf}\frac{1}{x^2}:&=\lim_{\epsilon\to0^{+}}\int\limits_{|x|>\epsilon}\frac{\varphi(x)}{x^2}\mathrm{d}x\,-2\frac{\varphi(0)}{\epsilon}
\end{split}
\end{equation*}
Where both distributions are linked by the following identity: 
\begin{equation}
\label{derivative_pv_1_x}
\mathscr{D}\,\mathrm{Vp}\frac{1}{x}=-\mathrm{Pf}\frac{1}{x^2}
\end{equation}
A simple proof of the latter formula is given in App.~\ref{app1}.
\end{exm}
\subsection{Logarithmic branch}
The integral Eq.~\ref{specific_distribution} is in general multi-valued because of the fractional power $x^{\alpha}$ occurring in the integrand. In order to avoid ambiguities of this kind it is necessary to fix the branch of the complex logarithm. It turns out that the following definition is the convenient choice. (See Eq.~\ref{prefac=1}).
\begin{align*}
\label{branch_log}
\log \colon \mathbb{C}^{\times}&\to
\mathbb{C}\\
r\mathrm{e}^{i\varphi} &\mapsto \log{r}+i\varphi
\end{align*}
where $\mathbb{C}^{\times}:=\{r\mathrm{e}^{i\varphi}|\,r>0,\,-\pi\leq\varphi<\pi \}$.\\
\begin{rem}
Note that our choice of the branch differs from the \emph{principal branch} on the negative half axis $\mathbb{R}^{-}:=\mathbb{R}\setminus\{x\in\mathbb{R}:x\leq0\}$. In our choice  we have \eg:
\begin{equation*}
 \sqrt{-1}=\mathrm{e}^{\frac{1}{2}\log{(-1)}}=-i
\end{equation*}
We mention also a few identities we use constantly throughout the paper:
\begin{equation*}
\log{i}=i\frac{\pi}{2}\mathrm{,} \quad \log{(-i)}=-i\frac{\pi}{2}\mathrm{,} \quad \log{(-1)}=-i\pi
\end{equation*}
\end{rem}
\section{\label{Sec.1}Fourier transforms of tempered distributions}
In the present section we evaluate the Fourier transform of a few distributions we need in Sec.~\eqref{sec_frac_deriv}. As is well known,
a tempered distribution $T[\varphi]\in \mathscr{S}^{\prime}\left(\mathbb{R}\right)$ has a well-defined Fourier transform which is given by:\\
\begin{equation*}
\mathscr{F}T[\varphi]:=T[\hat{\varphi}]\quad \mathrm{for}\,\, \mathrm{every}\,\, \mathrm{Schwartz} \,\,\mathrm{function}\,\varphi \in \mathscr{S}(\mathbb{R})
\end{equation*}\\
where we use the following convention for the Fourier transform:
\begin{equation*}
\hat{\varphi}(k):=\int\limits_{\mathbb{R}}\varphi(x)\,\mathrm{e}^{-ikx}\mathrm{d}x
\end{equation*}
\subsection{\label{subsec.A}Fourier transform of $\mathbf{\mathrm{Pf}\,x^{\alpha}}$ with real $\alpha$}
A function of the form $f(x):=x^{\alpha},\, \alpha\geq 0$ is polinomially bounded. It generates therefore a tempered distribution (where we use the logarithmic branch defined above): 
\begin{equation*} 
T_{f}[\varphi]=\int\limits_{\mathbb{R}}\varphi(x) f(x)\mathrm{d}x
\end{equation*}
with a well-defined Fourier transform $T_{f}[\hat{\varphi}]$. In the present section we show by applying Hadamard's regularization technique that the distributional Fourier transform can be extended to the general case of a real exponent $\alpha$.
\begin{lem}
\label{lemma}
\begin{equation*}
\mathscr{S}^{\prime}-\lim_{\epsilon\to0^{+}}\frac{1}{(\epsilon\pm ix)^{\alpha+1}}=
\begin{cases}
\frac{(\pm i)^{\alpha}}{\alpha!}\left[\pi\,\mathscr{D}^{\alpha}\delta\mp i\, \mathscr{D}^{\alpha}\,\mathrm{Vp}\,\frac{1}{x}\right]&\mathrm{if}\, \alpha\in \mathbb{N}_{0}\\
\mathrm{Pf}\,\frac{1}{(\pm ix)^{\alpha+1}}&\mathrm{if}\, \alpha\in \mathbb{R}\setminus \mathbb{N}_{0}\\
\end{cases}
\end{equation*}
where $\mathscr{D}^{n}\delta[\varphi]$ denotes the $n^{\mathrm{th}}$ derivative of the Dirac-distribution $\delta[\varphi]$.
\end{lem}
\begin{proof}
First, we prove the case $\alpha=n\in\mathbb{N}$.\\
\\
Let $\varphi \in \mathscr{S}(\mathbb{R})$. Repeated integration by parts gives then:\\
\begin{equation}
\label{sokhotski}
\begin{split}
\lim_{\epsilon\to0^{+}}\int\limits_{\mathbb{R}}\mathrm{d}x\,\varphi(x)(\epsilon\pm ix)^{-n-1}
&=\frac{(\mp i)^{n}}{n!}\lim_{\epsilon\to0^{+}}\int\limits_{\mathbb{R}}\mathrm{d}x\,\varphi^{(n)}(x)(\epsilon\pm ix)^{-1}
\end{split}
\end{equation}
Applying the well-known formula (Sokhotski–Plemelj) gives:\\
\begin{equation}
\label{Sokhotski_Plemelj}
\begin{split}
\lim_{\epsilon\to0^{+}}\int\limits_{\mathbb{R}}\mathrm{d}x\,\frac{\varphi(x)}{\epsilon\pm ix}=\pi\varphi(0)\mp i\,\mathrm{Vp}\frac{1}{x}
=\pi\delta\mp i\,\mathrm{Vp}\frac{1}{x}
\end{split}
\end{equation}
yields
\begin{equation*}
\begin{split}
\eqref{sokhotski}&=\frac{(\pm i)^{n}}{n!}\left[(-1)^{n}\,\pi\,\varphi^{(n)}(0)\mp i\,(-1)^{n}\, \mathrm{Vp}\left(\frac{1}{x}\right)[\varphi^{(n)}]\right]\\
&=\frac{(\pm i)^{n}}{n!}\left(\pi\,\mathscr{D}^{n}\delta\mp i\, \mathscr{D}^{n}\,\mathrm{Vp}\frac{1}{x}\right)
\end{split}
\end{equation*}
Next, we  prove the case: $\alpha\in \mathbb{R}\setminus \mathbb{N}_{0}$.
It is convenient to write $\alpha=[\alpha]+\gamma$ where $\gamma$ assumes a value between $0$ and $1$.
This leads to the expression:\\
\begin{equation*} 
\frac{1}{(\epsilon\pm ix)^{\alpha+1}}=\frac{1}{(\epsilon\pm ix)^{[\alpha]+1+\gamma}}
=\frac{(\pm i)^{[\alpha]+1}}{([\alpha]+1)!}\,\frac{\mathrm{d}^{[\alpha]+1}}{\mathrm{d}x^{[\alpha]+1}}\frac{1}{(\epsilon\pm ix)^{\gamma}}
\end{equation*}\\
Now, by Taylor expanding up to $1^{\mathrm{st}}$ order:\\
\begin{equation*}
\begin{split}
\frac{1}{(\epsilon\pm ix)^{\gamma}}=\frac{(\epsilon\pm ix)^{1-\gamma}}{(\epsilon\pm ix)}
=\frac{(1}{(\epsilon\pm ix)}\left[(\pm ix)^{1-\gamma}+\frac{1-\gamma}{(\pm ix)^{\gamma}}\epsilon+\mathcal{O}(\epsilon^2)\right]
\end{split}
\end{equation*}\\
In the limit $\epsilon\to0^{+}$ one gets for the first term:
\begin{equation}
\label{ableitung}
\mathscr{S}^{\prime}-\lim_{\epsilon\to0^{+}}\frac{(\pm ix)^{1-\gamma}}{\epsilon\pm ix}=(\pm ix)^{1-\gamma}\left[\pi\delta\mp i\,\mathrm{Vp}\frac{1}{x}\right]=\mathrm{Pf}\frac{1}{(\pm ix)^{\gamma}}
\end{equation}
where the identity $x^{1-\gamma}\,\delta[\varphi]=0$ has been used.
The $2^{\mathrm{nd}}$ term goes to zero in the limit $\epsilon\to0^{+}$ since:
\begin{equation*}
\mathscr{S}^{\prime}-\lim_{\epsilon\to0^{+}}\frac{\epsilon}{\epsilon\pm ix}=0
\end{equation*}\\
By taking the $([\alpha]+1)$th derivative of the right-hand-side of Eq.~\ref{ableitung} one obtains the claimed result.
\end{proof}
\begin{thm}
\label{theorem4}
Let $f(x)=x^{\alpha}$ (with $\alpha\in \mathbb{R}$) be the monomial which generates the regular distribution $\mathrm{Pf}\,x^{\alpha}$ then:\\
\begin{equation}
\label{formula_integer}
\mathscr{F}\,\mathrm{Pf}\,x^{\alpha}=\begin{cases}
2\pi \,i^{\alpha}\,\mathscr{D}^{\alpha}\delta[\varphi]&\text{if $\alpha\in\{0,1,2,3...\}$}\\
T_{g}[\varphi]&\text{if $\alpha\in\{-1,-2,-3...\}$}\\
T[\varphi]&\text{if $\alpha\in\mathbb{R}\setminus\mathbb{Z}$}
\end{cases}
\end{equation}
with 
\begin{equation}
\label{formula_power_negative}
g(x)=\pi\frac{i^{\alpha}\,x^{-\alpha-1}}{(-\alpha-1)!}\,\mathrm{sgn}{(x)}
\end{equation}
and
\begin{equation}
T=\Gamma(\alpha+1)\,\left(\mathrm{e}^{-i\pi\alpha}\,\mathrm{Pf}\,\frac{1}{(-ix)^{\alpha+1}}+\mathrm{Pf}\,\frac{1}{(ix)^{\alpha+1}}\right)
\end{equation}
$\Gamma(.)$ denotes the Gamma function defined in $\mathbb{C}\setminus\{0,-1,-2,-3...\}$.
\end{thm}
\begin{proof}
The case
{$\alpha\in\{-1,-2,-3...\}$} can be proven easily by starting with the well-known identity (corresponding to the value $\alpha=-1$):
\begin{equation}
\label{fourier_1_over_x}
\mathscr{F}\,\mathrm{Vp}\frac{1}{x}=-i\pi\,\mathrm{sgn}[\varphi]
\end{equation}
where the \emph{signum} distribution $\mathrm{sgn}[\varphi]$ is defined by:
\begin{equation*}
\mathrm{sgn}[\varphi]:=\int\limits_{\mathbb{R}}\varphi(x)\,\mathrm{sgn}(x)\mathrm{d}x
\end{equation*} 
A short proof of Eq.~\ref{fourier_1_over_x} can be found in App.~\ref{app1}.\\
\\
Next, we consider:
\begin{equation}
\label{insert_equation}
\mathscr{F}\mathscr{D}^{n}\mathrm{Vp}\,\frac{1}{x}=(ix)^n \mathscr{F}\,\mathrm{Vp}\,\frac{1}{x}
\end{equation}
By inserting the following well-known equation into Eq.~\ref{insert_equation}:
\begin{equation*}
\mathscr{D}^{n}\,\mathrm{Vp}\,\frac{1}{x}=(-1)^{n}\,n!\,\mathrm{Pf}\,\frac{1}{x^{n+1}}
\end{equation*}
one immediately obtains Eq.~\ref{formula_power_negative}.\\
In the next step we prove Theorem \ref{theorem4} for the case $\alpha \in\mathbb{R}\setminus\{-1,-2,\cdots\}$ essentially by applying the standard technique of analytic continuation.
\\
In a first step, we assume $\alpha$ to be larger than minus one. We consider then the following
well-defined (unambigous within the chosen logarithmic branch) classical Fourier transform:
\begin{equation}
\label{limit}
\mathscr{F}x^{\alpha}\mathrm{e}^{-\epsilon|x|}:=\int\limits_{\mathbb{R}}\mathrm{d}x\,x^{\alpha}\,\mathrm{e}^{-\epsilon|x|}\,\hat{\varphi}(x)
\end{equation}
Which equals:
\begin{equation}
\mathscr{F}x^{\alpha}\mathrm{e}^{-\epsilon|x|}=\int\limits_{\mathbb{R}}\mathrm{d}x\,x^{\alpha}\,\mathrm{e}^{-\epsilon|x|}\int\limits_{\mathbb{R}}\mathrm{d}k\,\varphi(k)\,\mathrm{e}^{-ikx}
\end{equation}
By Fubini's Theorem we get:
\begin{equation}
\label{int}
\mathscr{F}x^{\alpha}\mathrm{e}^{-\epsilon|x|}=\int\limits_{\mathbb{R}}\mathrm{d}k\,\varphi(k)\int\limits_{\mathbb{R}}\mathrm{d}x\,x^{\alpha}\,\mathrm{e}^{-ikx-\epsilon|x|}
\end{equation}
First, we evaluate the integral for the case where $\alpha$ is a positive integer $n$. Application of Cauchy's Residue Theorem gives:
\begin{equation*}
\int\limits_{\mathbb{R}}\mathrm{d}x\,x^{n}\,\mathrm{e}^{-ikx-\epsilon|x|}
=\Gamma(n+1)\,\left[\mathrm{e}^{-i\pi n}\,(\epsilon-ik)^{-n-1}+(\epsilon+ik)^{-n-1}\right]
\end{equation*}
Due to uniqueness of the analytic continuation one obtains a formula for the extended case $\alpha>-1$:
\begin{equation}
\label{express}
\int\limits_{\mathbb{R}}\mathrm{d}x\,x^{\alpha}\,\mathrm{e}^{-ikx-\epsilon|x|}=\Gamma(\alpha+1)\,\left[\mathrm{e}^{-i\pi\alpha}\,(\epsilon-ik)^{-\alpha-1}+(\epsilon+ik)^{-\alpha-1}\right]
\end{equation}
Next, we insert the right-hand side of Eq.~\ref{express} into Eq.~\ref{int} and perform the following limit: 
\begin{equation}
\label{schluss}
\lim_{\epsilon\to0^{+}}\mathscr{F}x^{\alpha}\mathrm{e}^{-\epsilon|x|}=\Gamma(\alpha+1)\lim_{\epsilon\to0^{+}}\int\limits_{\mathbb{R}}\mathrm{d}k\,\varphi(k)
\,\left[\mathrm{e}^{-i\pi\alpha}\,(\epsilon-ik)^{-\alpha-1}+(\epsilon+ik)^{-\alpha-1}\right]
\end{equation}
The limit in Eq.~\ref{schluss} can be evaluated by applying Lemma \ref{lemma}. The range of validity is (again by analytic continuation) extended to $\alpha \in\mathbb{R}\setminus\{-1,-2,\cdots\}$ which leads to the statement of Theorem \ref{theorem4}.
\end{proof}
A short calculation shows furthermore:
\begin{equation*}
\mathscr{F}\mathscr{F}\,\mathrm{Pf}\,x^{\alpha}=2\pi\,\mathrm{Pf}\,(-x)^{\alpha}
\end{equation*}
\begin{cor}
In the specific case of half-integer exponents: $\alpha=n+\frac{1}{2}$ with $n\in\mathbb{Z}$ one obtains the formula:
\begin{equation}
\label{formula_I}
\mathscr{F}\,\mathrm{Pf}\,x^{n+\frac{1}{2}}=(n+\frac{1}{2})\,\Gamma(n+\frac{1}{2})\left[-i\,(-1)^{n}\,\mathrm{Pf}\,(-ix)^{-n-\frac{3}{2}}+\mathrm{Pf}\,(ix)^{-n-\frac{3}{2}}\right]
\end{equation}
where for non-negative integer values of $n$ one can apply the well-known formulae:
\begin{equation*}
\begin{split}
\Gamma(n+\frac{1}{2})&=\frac{(2n!)}{4^{n}n!}\sqrt{\pi}\\
\Gamma(\frac{1}{2}-n)&=\frac{(-4)^{n}n!}{(2n)!}\sqrt{\pi}
\end{split}
\end{equation*}
An analogous calculation yields the following formula: 
\begin{equation}
\label{formula_II}
\mathscr{F}\mathrm{Pf}\,|x|^{n+\frac{1}{2}}=\Gamma(\frac{1}{2}-n)\left[\mathrm{Pf}\,(-ix)^{n-\frac{1}{2}}+\mathrm{Pf}\,(ix)^{n-\frac{1}{2}}\right]
\end{equation}
\end{cor}
In Table~\ref{tabelle} formulae \eqref{formula_I}, \eqref{formula_II} are evaluated for some values of $n$. Note that we use the convention: $-1!!=0!!:=1$. For brevity, the sign \emph{Pf} is omitted in front of the entries.
 \begin{table}[h!]
\renewcommand{\arraystretch}{1.8}  
	\centering
	\subfloat[FT of $x^{n+\frac{1}{2}}$]{
		\begin{tabular}{|c|c|} 
				 \hline
 $T$ & Fourier Transform \\ [1.5ex] 
 \hline\hline 
$x^{-\frac{7}{2}}$ & $-\frac{2^4\sqrt{\pi}}{5!!}H(-x)(-i|x|)^{+5/2}$\\
 \hline
$x^{-\frac{5}{2}}$ & $\frac{2^3\sqrt{\pi}}{3!!}H(-x)(-i|x|)^{+3/2}$\\ 
 \hline
$x^{-\frac{3}{2}}$ & $-\frac{2^2\sqrt{\pi}}{1!!}H(-x)(-i|x|)^{+1/2}$\\ 
 \hline
 $x^{-\frac{1}{2}}$ & $\frac{2^1\sqrt{\pi}}{0!!}H(-x)(-i|x|)^{-1/2}$\\ 
 \hline
 $x^{+\frac{1}{2}}$ & $\frac{1!!}{2^0\sqrt{\pi}}H(-x)(-i|x|)^{-3/2}$ \\
 \hline
 $x^{+\frac{3}{2}}$ & $\frac{3!!\sqrt{\pi}}{2^1}H(-x)(-i|x|)^{-5/2}$ \\
 \hline
 $x^{+\frac{5}{2}}$ & $\frac{5!!\sqrt{\pi}}{2^2}H(-x)(-i|x|)^{-7/2}$ \\
 \hline
 $x^{+\frac{7}{2}}$ & $\frac{7!!\sqrt{\pi}}{2^3}H(-x)(-i|x|)^{-9/2}$ \\
 \hline
	  \end{tabular}
	}	
	\qquad
	\subfloat[FT of $|x|^{n+\frac{1}{2}}$]{
		\begin{tabular}{|c|c|} 
				\hline
 $T$ & Fourier Transform \\ [1.5ex] 
 \hline\hline
$|x|^{-\frac{7}{2}}$ & $+\sqrt{2\pi}\frac{2^3}{5!!}\,|x|^{+5/2}$\\ 
 \hline
$|x|^{-\frac{5}{2}}$ & $-\sqrt{2\pi}\frac{2^2}{3!!}\,|x|^{+3/2}$\\ 
 \hline
$|x|^{-\frac{3}{2}}$ & $-\sqrt{2\pi}\frac{2^1}{1!!}\,|x|^{+1/2}$\\ 
 \hline
 $|x|^{-\frac{1}{2}}$ & $+\sqrt{2\pi}\frac{2^0}{(-1)!!}\,|x|^{-1/2}$\\ 
 \hline
 $|x|^{+\frac{1}{2}}$ & $-\sqrt{2\pi}\,\frac{1!!}{2^1}|x|^{-3/2}$ \\
 \hline
 $|x|^{+\frac{3}{2}}$ & $-\sqrt{2\pi}\,\frac{3!!}{2^2}|x|^{-5/2}$ \\
 \hline
 $|x|^{+\frac{5}{2}}$ & $+\sqrt{2\pi}\,\frac{5!!}{2^3}|x|^{-7/2}$ \\
 \hline
 $|x|^{+\frac{7}{2}}$ & $+\sqrt{2\pi}\,\frac{7!!}{2^4}|x|^{-9/2}$ \\
 \hline
    \end{tabular}
	}
	\caption{Distributional Fourier transform of (a) $x^{n+\frac{1}{2}}$, (b) $|x|^{n+\frac{1}{2}}$ }
\label{tabelle}
\end{table}
\subsection{Fourier transform of tempered distributions generated by rational functions}
Based on the explicit evaluation of $\mathscr{F}\,\mathrm{Pf}\frac{1}{x^n}$ ($n\in \mathbb{Z}$) one can immediately derive the distributional Fourier transform of a tempered distribution which is generated by an arbitrary rational function $R(x)=\frac{P_n(x)}{Q_m(x)}$ where  $P_n(x),Q_m(x)$ are Polynomials (with complex coefficients) of order $n$, m respectively.
Such a general statement goes beyond the possibility of classical analysis where the formulation is usually given in the following form:\\
\begin{rem}
Let $f(z)=R(z)\mathrm{e}^{-ikz}$ where $R(z)$ is a rational function with simple poles at
$p_1<\cdots<p_n \in \mathbb{R}$, $n\in \mathbb{N}_{0}$ and no further poles on $\mathbb{R}$.
Moreover, $\lim_{z\to\infty}R(z)=0$. Then
\begin{equation*}
\mathrm{Vp}\int\limits_{\mathbb{R}}R(x)\mathrm{e}^{-ikx}\mathrm{d}x=\pi i \sum_{l=1}^{n}\mathrm{Res}_{p_l}f(z)+
\begin{cases}
-2\pi i\sum_{\mathrm{Im}a>0}\mathrm{Res}_{a}f(z)&\text{if $k>0$}\\
2\pi i\sum_{\mathrm{Im}a>0}\mathrm{Res}_{a}f(z)&\text{if $k<0$}
\end{cases}
\end{equation*}
In the context of tempered distributions the following general statement is possible. 
\end{rem}
\begin{cor}
Let $R(z)=\frac{P(z)}{Q(z)}$ be a rational (complex-valued) function ($Q\neq0$, $\deg{P}=n$, $\deg{Q}=m$) with poles at $z_1,\,z_2,\,\dots, z_n \in \mathbb{C}$. The variable $m_i$ denotes the multiplicity of the corresponding pole $z_i$.
\\Let $R(x)=\frac{P(x)}{Q(x)}$ be the function which generates the distribution $T_{R}[\varphi]\equiv \mathrm{Pf}\,R$.\\
\\
Then
\begin{equation*}
\mathscr{F}T_{R}[\varphi]=\mathscr{F}T_{P^{\star}}[\varphi]+\sum_{i=1}^{n}\sum_{j=1}^{m_i}a_{ij}\,\mathscr{F}\mathrm{Pf}\,(x-z_{i})^{-j},\,\quad a_{ij}\in\mathbb{C}
\end{equation*}
where
\begin{equation*}
\mathscr{F}T_{P^{\star}}[\varphi]=2\pi\sum_{k=0}^{n-m}\,b_k\, i^k\mathscr{D}^{k}\delta[\varphi]\,,\quad b_k\in\mathbb{C}
\end{equation*}
is the Fourier transform of the Polynom:
\begin{equation*}
P^{\star}(x)=\sum_{k=0}^{n-m}b_k\, x^k\,,\quad b_k\in\mathbb{C}
\end{equation*}
and
\begin{equation*}
\mathscr{F}\mathrm{Pf}\,(x-z_{i})^{-j}=T_{g}[\varphi]
\end{equation*}
The function $g(x)$ in the latter expression is given by:
\begin{equation*}
\label{formula_g}
g(x)=\frac{2\pi i^j}{(j-1)!}\,|x|^{j-1}\,
\mathrm{e}^{-i\operatorname{Re}{(z_{i})}x}
\mathrm{e}^{-\operatorname{Im}{(z_{i})}|x|}\left[H(\operatorname{Im}{(z_{i})})H(-x)+(-1)^{j}H(-\operatorname{Im}{(z_{i})})H(x)\right]
\end{equation*}
Where $H(\cdot)$ denotes Heaviside's unit step function defined by:
\begin{equation*}
H(x)=\begin{cases}
1&\text{if} \quad x>0\\
\frac{1}{2}&\text{if} \quad x=0\\
0&\text{if} \quad x<0
\end{cases}
\end{equation*}
The expressions $\operatorname{Re}{(z_{i})}$, $\operatorname{Im}{(z_{i})}$ denote the real and imaginary part of $z_i$. 
\end{cor}
\begin{proof}
Through a polynomial division and partial fraction decomposition one may write:
\begin{equation}
\label{double_sum}
R(x)=P^{\star}(x)+\sum_{i=1}^{n}\sum_{j=1}^{m_i}\frac{a_{ij}}{(x-z_i)^j}\,,\quad a_{ij}\in\mathbb{C}
\end{equation}
where: 
\begin{equation*}
P^{\star}(x)=\sum_{k=0}^{n-m}b_k\, x^k\,,\quad b_k\in\mathbb{C}
\end{equation*}
Applying Theorem \ref{theorem4} gives the Fourier transform of the Polynom $P^{\star}$:
\begin{equation*}
\mathscr{F}T_{P^{\star}}[\varphi]=2\pi\sum_{k=0}^{n-m}b_k\, i^k\mathscr{D}^{k}\delta[\varphi]\,,\quad b_k\in\mathbb{C}
\end{equation*}
In the next step we have to evaluate the Fourier transform of the function $f(x)=(x-z_{i})^{-j}$ occurring in Eq.~\ref{double_sum}.\\
In order to do that it is convenient to distinguish between
$z_i\in \mathbb{R}$ and $z_i\in \mathbb{C}\setminus\mathbb{R}$.
\\
{case $z_i=x_i\in \mathbb{R}$}:\\
\\
A short calculation (using Eq.~\ref{formula_power_negative}) yields:
\begin{equation}
\label{hello}
\mathscr{F}\mathrm{Pf}\,(x-x_{i})^{-j}=\mathrm{e}^{-ixx_i}\mathscr{F}\mathrm{Pf}\,x^{-j}=\pi\frac{i^{-j}\,x^{j-1}}{(j-1)!}\,\mathrm{sgn}{(x)}\,\mathrm{e}^{-ixx_i}
\end{equation}
{case $z_i\in \mathbb{C}\setminus\mathbb{R}$}:\\
\\
The Fourier integral in the classical sense:
\begin{equation} 
\label{fourier_integral}
\hat{f}(k)=\int\limits_{\mathbb{R}}\frac{\mathrm{e}^{-ikx}}{(x-z_{i})^{j}}\mathrm{d}x
\end{equation}
is convergent for every pair $\{j,k\}\in \{\mathbb{N}\times\mathbb{R}\}\setminus\{1,0\}$. 
For $j\in\{2, 3,\cdots\}$ it is therefore possible to evaluate the Fourier transform \eqref{fourier_integral} classically. 
A standard calculation by using the Residue Theorem \cite{ahlfors:78} gives:
\\
\begin{equation}
\label{formula_rational}
\hat{f}(k):=\frac{2\pi i^j}{(j-1)!}|k|^{j-1}
\mathrm{e}^{-i\operatorname{Re}{(z_{i})}k}
\mathrm{e}^{-\operatorname{Im}{(z_{i})}|k|}\left[H(\operatorname{Im}{(z_{i})})H(-k)+(-1)^{j}H(-\operatorname{Im}{(z_{i})})H(k)\right]
\end{equation}
The case $j=1$ demands a separate discussion. The \emph{translation property} of Fourier transforms holds also for tempered distributions. A shift by an amount of $\operatorname{Re}{(z_{i})}$ gives:
\begin{equation*}
\mathscr{F}\mathrm{Pf}\,(x-z_{i})^{-1}=\mathrm{e}^{-i\,\operatorname{Re}{(z_{i})}}\mathscr{F}\mathrm{Pf}\,[x-i\operatorname{Im}{(z_{i})}]^{-1}
\end{equation*}
Rewriting the latter expression in integral form ($\mathrm{Pf}$ can be replaced by $\mathrm{Vp}$) yields: 
\begin{equation}
\label{case_j_1}
\mathscr{F}\mathrm{Pf}\,(x-z_{i})^{-1}=\mathrm{e}^{-i\,\operatorname{Re}{(z_{i})}}\,\mathrm{Vp}\int\limits_{\mathbb{R}}\frac{\hat{\varphi}(x)}{x-i\operatorname{Im}{(z_{i})}}\mathrm{d}x
\end{equation}
A short calculation by a reformulation of the integrand leads to the following expression:
\begin{equation*}
\mathscr{F}\mathrm{Pf}\,(x-z_{i})^{-1}
=\mathrm{e}^{-i\,\operatorname{Re}{(z_{i})}}\int\limits_{\mathbb{R}}\mathrm{d}k\,{\varphi}(k)\,\mathrm{Vp}\int\limits_{\mathbb{R}}\frac{x+i\operatorname{Im}{(z_{i})}}{x^2+|z_{i}|^2}\mathrm{e}^{-ikx}\mathrm{d}x
\end{equation*}
A short calculation (again, using contour integration) yields:
\begin{equation*}
\mathrm{Vp}\int\limits_{\mathbb{R}}\frac{x+i\operatorname{Im}{(z_{i})}}{x^2+|z_{i}|^2}\mathrm{e}^{-ikx}\mathrm{d}x=2\pi i\,
\mathrm{e}^{-\operatorname{Im}{(z_{i})}|k|}\left[H(\operatorname{Im}{(z_{i})})H(-k)-H(-\operatorname{Im}{(z_{i})})H(k)\right]
\end{equation*}
which is identical to formula \ref{formula_rational} for $j=1$. Hence, Formula       
 \ref{formula_rational} covers also the case $j=1$.
Moreover, the function $g(x)$ coincides with the right-hand side of Eq.~\ref{hello} even in the case $\Im(z_i)=0$. The function $g(x)$ represents therefore the correct generator for $\mathscr{F}\mathrm{Pf}\,(x-z_{i})^{-j}$ for every complex pole $z_{i}\in \mathbb{C}$.
\end{proof} 
\begin{table}[h!]
\renewcommand{\arraystretch}{1.8}  
\centering

		\begin{tabular}{|c|c|} 
				 \hline
 Distribution & Fourier Transform \\ [1.5ex] 
 \hline\hline
 $(x^2-a^2)^{-1}$ & $-\frac{\pi}{a}\,\text{sgn}(ax)\sin{ax}$\\ 
 \hline
 $(x-1)^{-2}$ & $-\pi|x|\,\mathrm{e}^{-ix}$\\ 
 \hline
 $(x-i)^{-1}$& $2\pi i H(-x)\,\mathrm{e}^{-|x|}$\\ 
 \hline
 $\frac{x^2}{x-i}$& $2\pi i\,\delta^{\prime}+2\pi i\,\delta-2\pi i\,H(-x)\,\mathrm{e}^{-|x|}$\\ 
 \hline
	  \end{tabular}\\
\caption{FT of some rational distributions}
\end{table}
\subsection{Distributional Fourier transform of quantum statistics functions}
The present section is entirely expository. The results can be used for didactic purposes. By means of Hadamards - finite part we evaluate the distributional Fourier transform of functions well-known as Fermi-Dirac and Bose-Einstein quantum statistics. For simplicity, we assume a vanishing chemical potential: $\mu=0$. These functions read then:
\begin{equation}
\label{intro_ex}
f_{\pm}(x)=\frac{1}{\mathrm{e}^{\beta x}\pm1},\quad \beta\in\mathbb{R_{+}}
\end{equation}
Where ($+$) corresponds to Fermi-Dirac and ($-$) to Bose-Einstein. These functions can be turned into tempered distributions via the usual procedure. Its Fourier transforms are formally defined by:
\begin{equation}
\label{quantum_stat}
\mathscr{F}T_{f_{\pm}}[\varphi]:=\mathrm{Pf}\int\limits_{\mathbb{R}}\frac{\hat{\varphi}(x)}{\mathrm{e}^{\beta x}\pm1}\,\mathrm{d}x:=\mathrm{f.\,p.\,of}\lim_{\epsilon\to0^{+}}\int\limits_{|x|>\epsilon}\frac{\hat{\varphi}(x)}{\mathrm{e}^{\beta x}\pm1}\,\mathrm{d}x
\end{equation}
An explicit evaluation of Eq.~\ref{quantum_stat} is then given by the following theorem.
\begin{thm}
\label{Theorem_quantum_statistics}
\begin{equation}
\label{formula_quantum_stat}
\mathscr{F}T_{f_{\pm}}[\varphi]=\pm\pi\delta[\varphi]+\begin{cases}
T_{g}[\varphi]&(+)\\
T_{h}[\varphi]&(-)\\
\end{cases}
\end{equation}
where 
\begin{equation}
g(x)=\frac{i\pi}{\beta\sinh{\frac{\pi x}{\beta}}}\quad\mathrm{and}\quad h(x)=\frac{\pi}{i\beta}\coth{\frac{\pi x }{\beta}}
\end{equation}
\end{thm}
\begin{proof}
\begin{equation}
\begin{split}
\label{evaluation}
\mathrm{Pf}\int\limits_{\mathbb{R}}\frac{\hat{\varphi}(x)}{\mathrm{e}^{\beta x}\pm1}\,\mathrm{d}x&:=\mathrm{Pf}\int\limits_{\mathbb{R}}\frac{\hat{\varphi}(x)}{\mathrm{e}^{\beta x}\pm1}\,\mathrm{e}^{-\epsilon|x|}\,\mathrm{d}x\\
&=\mathrm{Pf}\int\limits_{\mathbb{R}}\mathrm{d}x\,\frac{\hat{\varphi}(x)}{\mathrm{e}^{\beta x}\pm1}\,\mathrm{e}^{-\epsilon|x|}\,\int\limits_{\mathbb{R}}\mathrm{d}k\,\varphi(k)\,\mathrm{e}^{-ikx}
\end{split}
\end{equation}
Applying Fubini's theorem gives:
\begin{equation}
\label{evaluation2}
\mathscr{F}T_{f_{\pm}}[\varphi]:=\mathrm{f.\,p.\,of}\lim_{\epsilon\to0^{+}}\int\limits_{\mathbb{R}}\mathrm{d}k\,\varphi(k)\,\int\limits_{|x|>\epsilon}\mathrm{d}x\,\frac{\mathrm{e}^{-\epsilon|x|-ikx}}{\mathrm{e}^{\beta x}\pm1}
\end{equation}
The integral on the right-hand side of Eq.~\ref{evaluation2} can be rewritten as
\begin{equation*}
\pm\int^{\infty}_{\epsilon}\mathrm{d}x\frac{\mathrm{e}^{ikx-\epsilon x}}{1\pm\mathrm{e}^{-\beta x}}+\int^{\infty}_{\epsilon}\mathrm{d}x\frac{\mathrm{e}^{-ikx-(\beta+\epsilon)x}}{1\pm \mathrm{e}^{-\beta x}}
\end{equation*}
By expanding both denominators into a geometric series:
\begin{equation*}
\frac{1}{1\pm \mathrm{e}^{-\beta x}}=\sum^{\infty}_{n=0}(\mp1)^{n}\mathrm{e}^{-\beta nx}
\end{equation*}
one gets (after an interchange of the integral and the sum):
\begin{equation}
\label{sum}
\pm\int^{\infty}_{\epsilon}\mathrm{d}x\,\mathrm{e}^{ikx-\epsilon x}+ \sum^{\infty}_{n=1}(\mp1)^{n}\int^{\infty}_{\epsilon}\mathrm{d}x\left[\pm\mathrm{e}^{-x(\beta n+\epsilon-ik)}\mp\mathrm{e}^{-x\left(\beta n+\epsilon+ik\right)}\right]
\end{equation}
Evaluating both integrals yields:
\begin{equation}
\label{expression}
\eqref{sum}=\pm\frac{1}{\epsilon-ik}+\sum^{\infty}_{n=1}(\mp1)^{n}\left[\pm\frac{\mathrm{e}^{-\epsilon(\beta n+\epsilon-ik)}}{\beta n+\epsilon-ik}\mp\frac{\mathrm{e}^{-\epsilon(\beta n+\epsilon+ik)}}{\beta n+\epsilon+ik}\right]
\end{equation}
Plugging Eq.~\ref{expression} into Eq.~\ref{evaluation2} gives:
\begin{equation}
\label{result}
\begin{split}
\mathscr{F}T_{f_{\pm}}[\varphi]=&\pm\lim_{\epsilon\to0^{+}}\int\limits_{\mathbb{R}}\mathrm{d}k\,\frac{\varphi(k)}{\epsilon-ik}\\
&\pm\int\limits_{\mathbb{R}}\mathrm{d}k\,\varphi(k)\sum^{\infty}_{n=1}(\mp1)^{n}\frac{2ik}{\beta^2n^2+k^2}
\end{split}
\end{equation}
The limit in Eq.~\ref{result} is given by formula \ref{Sokhotski_Plemelj} (Sokhotski-Plemelj):
\begin{equation}
\label{sokhotski_2}
\pm\lim_{\epsilon\to0^{+}}\int\limits_{\mathbb{R}}\mathrm{d}x\,\frac{\varphi(x)}{\epsilon-ix}=
\pm\pi\delta[\varphi]\pm i\,\mathrm{Vp}\frac{1}{x}
\end{equation}
whereas the sum in Eq.~\ref{result} is determined by the following formula \cite{gradshteyn:2014}:
\begin{equation*}
\sum^{\infty}_{n=1}\frac{(-1)^n\cos{(nx)}}{\alpha^{2}-n^2}=\left(\frac{\pi\cos{\alpha x}}{\sin{\alpha\pi}}-\frac{1}{\alpha}\right)\frac{1}{2\alpha},\quad \alpha \in \mathbb{C/Z}
\end{equation*}
at $\alpha=\frac{i k}{\beta}$ and $x=0\, (+), x=n\pi\, (-)$, respectively. This gives then:
\begin{equation}
\label{sum2}
2ik\sum^{\infty}_{n=1}\frac{(\mp1)^{n}}{\beta^2n^2+k^2}=\begin{cases}
\frac{1}{ik}+\frac{i\pi}{\beta\sinh{\frac{k\pi}{\beta}}}&(+)\\
\frac{1}{ik}+\frac{i\pi}{\beta}\coth{\frac{k\pi}{\beta}}&(-)\\
\end{cases}
\end{equation}
Eqs.~\ref{sum2}, \ref{sokhotski_2} complete the proof.
\end{proof}
Theorem \ref{Theorem_quantum_statistics} enables us to evaluate the Fourier transform of Heaviside's unit step function $H$ (intended as a distribution $T_H$) in an alternative way.
The following limit is easy to check:
\begin{equation*}
\mathscr{S}^{\prime}-\lim_{\beta\to \infty}T_{f_{\pm}}[\varphi]=
\pm T_{H(-x)}
\end{equation*}
The operator $\mathscr{F}$ is continous in $\mathscr{S}^{\prime}\left(\mathbb{R}\right)$. One gets therefore:
\begin{equation}
\label{formule_1}
\mathscr{S}^{\prime}-\lim_{\beta\to \infty}\mathscr{F}T_{f_{\pm}}[\varphi]=\mathscr{F}\left[\mathscr{S}^{\prime}-\lim_{\beta\to \infty}T_{f_{\pm}}[\varphi]\right]=\pm\mathscr{F}T_{H(-x)}
\end{equation}
On the other hand, the limit $\beta \rightarrow \infty$ of Eq.~\ref{formula_quantum_stat} is given by the expression:
\begin{equation}
\label{formule_2}
\lim_{\beta\to \infty}T_{f_{\pm}}[\varphi]=\pm\pi\delta[\varphi]\pm i\,\mathrm{Vp}\frac{1}{x}
\end{equation} 
Comparison of Eqs.~\ref{formule_1}, \ref{formule_2} leads to the well-known expression:
\begin{equation*}
\mathscr{F}T_{H}[\varphi]=\pi\delta[\varphi]-i\,\mathrm{Vp}\frac{1}{x}[\varphi]
\end{equation*}
\section{Fractional derivatives of tempered distributions}
\label{sec_frac_deriv}
Tempered distributions $T\in \mathscr{S^{\prime}(\mathbb{R})}$ can be differentiated arbitrarily many times. Equally important is the fact that they have a well-defined Fourier transform.
These advantageous features can be used to define fractional derivatives of tempered distributions. 
In order to achieve such a goal we have to extend slightly the definition of multiplying a function with a distribution. Schwartz showed the impossibility of multiplying two distributions in the usual sense \cite{schwartz:1951}. As known, multipliying a tempered distribution with a smooth function $f(x) \in C^{\infty}(\mathbb{R})$ is possible.  For our purpose we need a definition where the function $f(x)$ is allowed to be non-differentiable at some specific points $x\in \mathbb{R}$.\\
More specifically, we want to multiply the monomial $x^{\alpha}$ ($\alpha \in \mathbb{R}$) with a tempered distribution $T[\varphi]$. In the following we introduce a few definitions we need to define fractional derivatives.
\begin{dfn}
\label{multiplications_distribution}
Let $T[\varphi] \in \mathscr{S}^{\prime}(\mathbb{R})$. For a real exponent $\alpha$ we define formally:
\\
\begin{equation}
\label{first_def}
\mathscr{F}\mathscr{D}^{\alpha}T[\varphi]:=i^{\alpha}x^{\alpha}\mathscr{F}T[\varphi]
\end{equation}
\end{dfn}
Where the monomial $x^{\alpha}$ must be interpreted as a \emph{function}.
However, the right-hand side of Eq.~\ref{first_def} must be given a proper meaning since in general the function $x^{\alpha}$ lacks smoothness. In order to achieve a suitable definition one has to distinguish according to whether $T[\varphi]$ is \emph{regular} or \emph{singular}. 
\begin{dfn}
\label{firstdefinition}
Let $T_f[\varphi]$ be a \emph{regular} distribution generated by $f(x)$. For $\alpha \in \mathbb{R}$ we define then:
\begin{equation*}
x^{\alpha}T_{f}[\varphi]:=T_{x^{\alpha}f}[\varphi]\in \mathscr{S}^{\prime}(\mathbb{R})
\end{equation*}
where:
\begin{equation*}
T_{x^{\alpha}f}[\varphi]:=\mathrm{Pf}\int\limits_{\mathbb{R}}x^{\alpha}f(x)\varphi(x)\mathrm{d}x
\end{equation*}
\end{dfn}
\begin{dfn}
\label{seconddefinition}
Let $T[\varphi]:=\mathscr{D}^{n}\delta[\varphi]$ be \emph{singular} then:
\begin{equation}
\label{muliplication_sing_distribution}
x^{\alpha}\mathscr{D}^{n}\delta[\varphi]:=\begin{cases}
0&\text{if $\alpha>n$}\\
(-1)^{n}\,n!\,\delta[\varphi]&\text{if $\alpha=n$}\\
(-1)^{\alpha}\frac{i^{n-\alpha}\Gamma(n+1)}{\Gamma(n-\alpha+1)}\mathscr{F}^{-1}\left[x^{n-\alpha}T_{1}\right]&\text{if $\alpha<n$}
\end{cases}
\end{equation}
\end{dfn}
The case $\alpha<n$ in Eq.~\ref{muliplication_sing_distribution} is justified by the following considerations. One starts with the identity:
\begin{equation}
\label{derivative_delta_dirac}
\mathscr{D}^{n}\delta[\varphi]=
(-1)^{n}n!\frac{\delta[\varphi]}{x^{n}}
\end{equation}
Multiplication with $x^{\alpha}$ gives:
\begin{equation}
\label{frac_delta}
x^{\alpha}\mathscr{D}^{n}\delta[\varphi]=
(-1)^{n}n!\frac{\delta[\varphi]}{x^{n-\alpha}}
\end{equation}
At this stage we introduce a somewhat heuristic argument. By presuming that the classical formula Eq.~\ref{derivative_delta_dirac}
holds also for fractional powers of the operator $\mathscr{D}$ one may write formally:
\begin{equation*}
\mathscr{D}^{n-\alpha}\,\delta\left[\varphi\right]=(-1)^{n-\alpha}\,\Gamma(n-\alpha+1)\frac{\delta[\varphi]}{x^{n-\alpha}}
\end{equation*}
Then we proceed by rewriting the right-hand side of Eq.~\ref{frac_delta} in the form:
\begin{equation}
\label{new_expression}
x^{\alpha}\mathscr{D}^{n}\delta[\varphi]=(-1)^{\alpha}\frac{\Gamma(n+1)}{\Gamma(n-\alpha+1)}\mathscr{D}^{n-\alpha}\,\delta\left[\varphi\right]
\end{equation}
By applying the identity $\mathds{1}=\mathscr{F}^{-1}\mathscr{F}$ to the right-hand side of Eq.~\ref{new_expression} and using $\mathscr{F}\delta[\varphi]=T_{1}[\varphi]$ one obtains:
\begin{equation*}
x^{\alpha}\mathscr{D}^{n}\delta[\varphi]=(-1)^{\alpha}i^{n-\alpha}\frac{\Gamma(n+1)}{\Gamma(n-\alpha+1)}\mathscr{F}^{-1}\left[x^{n-\alpha}T_{1}\right]\in \mathscr{S}^{\prime}(\mathbb{R})
\end{equation*}
Defs.~\ref{firstdefinition}, \ref{seconddefinition} suggest:
\begin{equation*}
x^{\alpha}\,T[\varphi]\in \mathscr{S}^{\prime}(\mathbb{R}),\quad \mathrm{for\, every}\,\, T[\varphi]\in \mathscr{S}^{\prime}(\mathbb{R})
\end{equation*}
\\
These preparatives enable us to define the operator $\mathscr{D}^{\alpha}$, as shown in the following.
\\
\begin{dfn}
According to the following diagram:
\\
\begin{center}
$\begin{CD}
T@>\mathscr{D}^{\alpha}>{\phantom{\text{very long label}}}>\mathscr{D}^{\alpha}T  \\
@V{\mathscr{F}}VV@AA{\mathscr{F}^{-1}}A  \\
\hat{T} @>\cdot(ix)^{\alpha}>{\phantom{\text{very long label}}}> (ix)^{\alpha}\hat{T}
\end{CD}$
\end{center}
we introduce for every real value $\alpha \in \mathbb{R}$ a linear operator $\mathscr{D}^{\alpha}$:
\begin{equation} 
\label{operator_fractional_derivative}
\mathscr{D}^{\alpha}\colon\mathscr{S^{\prime}(\mathbb{R})}\to \mathscr{S^{\prime}(\mathbb{R})},\quad \alpha \in \mathbb{R}.
\end{equation}
defined by:
\begin{equation} 
\label{the_definition}
\mathscr{D}^{\alpha}T[\varphi]:=\mathscr{F}^{-1}(ix)^{\alpha}\mathscr{F}T[\varphi]
\end{equation}
The operator $\mathscr{D}^{\alpha}$ is called the operator of fractional derivation of real order $\alpha$.
\subsection{algebraic properties}
It is immediately evident that within $\{\mathscr{D}^{\alpha}\}_{\alpha\geq 0}$ and $\{\mathscr{D}^{\alpha}\}_{\alpha\leq 0}$ the following is valid:
\begin{equation} 
\label{group_property}
\begin{split}
\mathscr{D}^{\alpha+\beta}T[\varphi]:&=\mathscr{F}^{-1}(ix)^{\alpha+\beta}\mathscr{F}T
=\mathscr{F}^{-1}(ix)^{\alpha}(ix)^{\beta}\mathscr{F}T\\
&=\mathscr{F}^{-1}(ix)^{\alpha}\mathscr{F}\mathscr{F}^{-1}(ix)^{\beta}\mathscr{F}T
=\mathscr{D}^{\alpha}\mathscr{D}^{\beta}T[\varphi].
\end{split}
\end{equation}
Which is the fundamental property of an abelian semigroup. The (unique) identity operator is given by $\mathds{1}:=\mathscr{D}^{0}$.\\
\begin{exm}
\begin{equation*}
\mathscr{D}^{\frac{1}{2}}\mathscr{D}^{\frac{1}{2}}x^{-\frac{1}{2}}=\mathscr{D}\,x^{-\frac{1}{2}}=-\frac{1}{2}\,x^{-\frac{3}{2}}
\end{equation*}
In a first step, we calculate the fractional derivative $\mathscr{D}^{\frac{1}{2}}\,x^{-\frac{1}{2}}$ which by definition is:
\begin{equation*}
\mathscr{F}^{-1}\sqrt{ix}\,\mathscr{F}x^{-\frac{1}{2}}
\end{equation*}
The Fourier transform of $x^{-\frac{1}{2}}$ can be looked up in Tab.~\ref{tabelle}. We obtain then:
\begin{equation*}
\mathscr{D}^{\frac{1}{2}}x^{-\frac{1}{2}}=\mathscr{F}^{-1}\sqrt{ix}\,\left[2\sqrt{\pi}\,H(-x)(-i|x|)^{-\frac{1}{2}}\right]=2\sqrt{\pi}\,\mathscr{F}^{-1}H(-x)
\end{equation*}
Since:
\begin{equation*}
\mathscr{F}^{-1}H(-x)=\frac{\delta[\varphi]}{2}-\frac{i}{2\pi x}
\end{equation*}
we get:
\begin{equation*}
\mathscr{D}^{\frac{1}{2}}x^{-\frac{1}{2}}=2\sqrt{\pi}\,\delta[\varphi]-\frac{i}{\sqrt{\pi}x}
\end{equation*}
Applying again a "half-derivative" $\mathscr{D}^{\frac{1}{2}}$ gives the claimed result:
\begin{equation}
\label{last_step}
2\mathscr{F}^{-1}\,\sqrt{-i\pi|x|}\,H(-x)=-\frac{1}{2}\,x^{-\frac{3}{2}}
\end{equation}
Where again, Tab.~\ref{tabelle} has been consulted for Eq.~\ref{last_step}.\\
\end{exm}
When one considers the "full" set $\{\mathscr{D}^{\alpha}\}_{\alpha \in \mathbb{R}}$ the commutative property is lost in general. The inverse operator exists but in general the order of multiplication matters as shown by the following simple example.
\begin{exm}
\begin{equation}
\label{non-commutative}
\mathscr{D}\mathscr{D}^{-1}H\neq \mathscr{D}^{-1}\mathscr{D}H
\end{equation}
The left-hand side of Eq.~\ref{non-commutative} is given by:
\begin{equation}
\begin{split}
\mathscr{D}\mathscr{D}^{-1}H:&=\mathscr{D}\left[\mathscr{F}^{-1}\frac{1}{ix}(\pi\,\delta-\frac{i}{x})\right]
=\mathscr{D}\left[\mathscr{F}^{-1}\left(i\pi\delta^{\prime}-\frac{1}{x^2}\right)\right]\\
&=\frac{1}{2}\mathscr{D}\left[x+|x|\right]=\mathscr{D}[xH]=1\cdot H+x\cdot\delta=H
\end{split}
\end{equation}
On the other hand, we obtain for the right-hand side of Eq.~\ref{non-commutative}:
\begin{equation}
\begin{split}
\mathscr{D}^{-1}\mathscr{D}H&=\mathscr{D}^{-1}\delta:=\mathscr{F}^{-1}\frac{1}{ix}\mathscr{F}\delta\\
&=\mathscr{F}^{-1}\mathrm{Vp}\,\frac{1}{ix}=\frac{1}{2}\mathrm{sgn}
\end{split}
\end{equation}
\end{exm}
We note that the commutative property (and with it the group property) still holds for the subspace of distributions in $\mathcal{S^{\prime}(\mathbb{R})}$ whose Fourier transforms are not singular:
\begin{equation*}
\mathscr{D}^{\alpha}\mathscr{D}^{\beta}T=\mathscr{D}^{\beta}\mathscr{D}^{\alpha}T=\mathscr{D}^{\alpha+\beta}T\quad \text{if}\,\, \mathscr{F}T\,\,\text{is}\,\,\text{not}\,\,\text{singular.} 
\end{equation*}
\\
The following examples are verified easily by using the results of Sec.~\ref{rational_functions}:
\begin{exm}
\begin{equation*}
\begin{split}
\mathscr{D}\mathscr{D}^{-1}\mathrm{Pf}\,\frac{1}{x}&=\mathscr{D}^{-1}\mathscr{D}\,\mathrm{Pf}\,\frac{1}{x}\\
\mathscr{D}^{-2}\mathscr{D}\,\frac{1}{x}&=\mathscr{D}\mathscr{D}^{-2}\,\frac{1}{x}
\end{split}
\end{equation*}
\end{exm}
\begin{rem}
The reason behind the non-commutativity within $\{\mathscr{D}^{\alpha}\}_{\alpha \in \mathbb{R}}$ is the occurrence of the product 
$(ix)^{\alpha}\,(ix)^{\beta}\,\mathscr{F}T$ in Eq.~\ref{group_property} which is not associative in general. Associativity holds only in the case when the exponents $\alpha$ and $\beta$ are of the same sign. An example of this fact is due to L. Schwartz \cite{schwartz:1951}, given as follows:
\begin{exm}
\begin{equation*}
\left(\frac{1}{x}\cdot x\right)\delta=1\cdot\delta=\delta\neq \frac{1}{x}(x\cdot \delta)=\frac{1}{x}\cdot 0=0
\end{equation*}
The identity $\frac{1}{x}\cdot x=1$ is not so trivial as it appears. A rigorous proof can be found in \cite{mikusinski:1973}.
\end{exm}
\end{rem} 
\end{dfn}
\subsection{Examples}
\label{examples_FD}
This section is devoted to the analysis of some typical examples. According to Def.~\ref{the_definition} we apply always the following calculational scheme:
\begin{equation*}
\mathscr{D}^{\alpha}:T\longrightarrow\mathscr{F}T\longrightarrow (ix)^{\alpha}\mathscr{F}T\longrightarrow \mathscr{F}^{-1}(ix)^{\alpha}\mathscr{F}T
\end{equation*}
\subsubsection{Periodic functions}
We want to evaluate the fractional derivative of the periodic distribution $\mathrm{e}^{iax}$. 
For notational brevity, we write $\mathrm{e}^{iax}$ instead of $T_{\mathrm{e}^{iax}}[\varphi]$.\\ 
In a first step we get:
\begin{equation}
\label{equation}
\mathscr{F}\mathscr{D}^{\alpha}\mathrm{e}^{iax}=2\pi\,i^{\alpha}x^{\alpha}\,\delta_{a}[\varphi]
=2\pi\,i^{\alpha}a^{\alpha}\,\delta_{a}[\varphi]
\end{equation}
where $\delta_a[\varphi]$ stands for the translated delta distribution $\delta_{a}[\varphi]:=\varphi(a)$.\\
Fourier inversion of Eq.~\ref{equation} yields:  
\begin{equation}
\label{half_derivative}
\mathscr{D}^{\alpha}\mathrm{e}^{iax}=\mathscr{F}^{-1}\left[i^{\alpha}x^{\alpha}2\pi\,\delta_{a}\right]=i^{\alpha}a^{\alpha}\mathrm{e}^{iax}
\end{equation}
In particular, for $\alpha=\frac{1}{2}$:
\begin{equation*}
\mathscr{D}^{\frac{1}{2}}\mathrm{e}^{iax}=\sqrt{\frac{a}{2}}(1+i)(\cos{ax}+i\sin{ax})
\end{equation*}
due to linearity of the operator $\mathscr{D}^{\frac{1}{2}}$ we obtain also:  
\begin{equation*}
\begin{split}
\mathscr{D}^{\frac{1}{2}}\cos{ax}&=\sqrt{\frac{a}{2}}(\cos{ax}-\sin{ax})\\
\mathscr{D}^{\frac{1}{2}}\sin{ax}&=\sqrt{\frac{a}{2}}(\cos{ax}+\sin{ax})
\end{split}
\end{equation*}
Thus, one gets: 
\begin{equation}
\label{half_deriv_sin_cos}
\begin{split}
&\mathscr{D}^{\frac{1}{2}}\mathscr{D}^{\frac{1}{2}}\cos{ax}=-\sin{ax}=\mathscr{D}\cos{ax}\\
&\mathscr{D}^{\frac{1}{2}}\mathscr{D}^{\frac{1}{2}}\sin{ax}=+\cos{ax}=\mathscr{D}\sin{ax}\\
\end{split}
\end{equation}
Eq.~\ref{half_derivative} offers the possibility to define fractional derivatives of Fourier series.\\
\subsubsection{Exponential $\mathrm{e}^{x}$}
Classical analysis shows that the exponential function $f(x)=\mathrm{e}^{x}$ is the only non-trivial
function whose derivative is equal to itself. It is interesting to see what happens when we evaluate the fractional derivative of the exponential. By the usual procedure we obtain:
\begin{equation*} 
\mathscr{D}^{\alpha}\mathrm{e}^{x}=\mathscr{F}^{-1}(ix)^{\alpha}\mathscr{F}\mathrm{e}^{x}
\end{equation*}
According to \cite{gelfand:1964} we have:
\begin{equation}
\label{caution}
\mathscr{F}\mathrm{e}^{x}=2\pi\,\delta_{i}[\varphi]
\end{equation}\\
Where $\delta_{i}[\varphi]$ denotes the translated ($i=$imaginary unit) delta distribution $\delta_{i}[\varphi]:=\varphi(i)$\\
Formula \ref{caution} has to be taken with some caution. The exponential $\mathrm{e}^{x}$ is not a tempered distribution. Despite this fact, the authors in \cite{gelfand:1964} derive its Fourier transform by transforming term-bye-term the power series of $\mathrm{e}^{x}$.\\
\\
Assuming the validity of Eq.~\ref{caution} we obtain:
\begin{equation*} 
\mathscr{D}^{\alpha}\mathrm{e}^{x}=2\pi\,\mathscr{F}^{-1}(ix)^{\alpha}\,\delta_{i}[\varphi]
\end{equation*}\\
By virtue of the general equation $g(x)\delta(x-a)=g(a)\delta(x-a)$ we obtain formally:\\
\begin{equation*}
\mathscr{D}^{\alpha}\mathrm{e}^{x}=2\pi\,\mathscr{F}^{-1}\delta_{i}[\varphi]
\end{equation*}
By the translation property:
\begin{equation*}
 \mathscr{F}^{-1}\delta_{i}[\varphi]=\mathrm{e}^{ix(-i)}\mathscr{F}^{-1}\delta[\varphi]
\end{equation*}
we get the remarkable formula:
\begin{equation*}
\mathscr{D}^{\alpha}\mathrm{e}^{x}=\mathrm{e}^{x}\quad\text{for \underline{every} real }\alpha \in \mathbb{R}
\end{equation*}
\subsubsection{Dirac Distribution and Heaviside unit step function}
First, we evaluate the fractional derivative of Heaviside's step function $H(x)$:
\begin{equation*}
\begin{split}
&\mathscr{F}\mathscr{D}^{\frac{1}{2}}T_H=\sqrt{ix}\,\mathscr{F}T_H=\sqrt{ix}\,\left[\pi\delta[\varphi]-i\,\mathrm{Vp}\frac{1}{x}\right]=\mathrm{Pf}\frac{1}{\sqrt{ix}}
\end{split}
\end{equation*}
According to Eq.~\ref{formula_I} one obtains:
\begin{equation}
\label{half_Heaviside}
\mathscr{D}^{\frac{1}{2}}T_H\equiv\mathscr{F}^{-1}\mathrm{Pf}\frac{1}{\sqrt{ix}}=\mathrm{Pf}\,\frac{H(x)}{\sqrt{\pi x}}\\
\end{equation}
It is interesting to note that Eq.~\ref{half_Heaviside} coincides with a corresponding formula given by Heaviside in \cite{heaviside:1970}(See also \cite{courant:1943}) where he used this formula in the context of operational calculus. 
However, this formula was actually known long before Heaviside \cite{ross:1977}.\\
In the same manner we obtain:
\begin{equation*}
\begin{split}
\mathscr{F}\mathscr{D}^{\frac{1}{2}}\delta[\varphi]=\sqrt{ix}\,\mathscr{F}\delta[\varphi]=\mathrm{Pf}\sqrt{ix}
\end{split}
\end{equation*}
Thus, by Eq.~\ref{formula_integer}: 
\begin{equation}
\label{half_delta}
\mathscr{D}^{\frac{1}{2}}\delta\equiv\mathscr{F}^{-1}\mathrm{Pf}\sqrt{ix}=-\mathrm{Pf}\,\frac{H(x)}{2\sqrt{\pi}\,|x|^{\frac{3}{2}}}
\end{equation}
Alternatively, the same result can be achieved in the following way:
\begin{equation*}
\mathscr{D}^{\frac{1}{2}}\,\delta[\varphi]=\mathscr{D}^{\frac{1}{2}}\mathscr{D}\,T_{H}[\varphi]
=\mathscr{D}\mathscr{D}^{\frac{1}{2}}\,T_{H}[\varphi]
\end{equation*}
By using the result Eq.~\ref{half_Heaviside} one gets Eq.~\ref{half_delta}, as expected.
Moreover, a short calculation yields:
\begin{equation*}
\begin{split}
\mathscr{D}^{\frac{1}{2}}\mathscr{D}^{\frac{1}{2}}\,\delta[\varphi]&=\mathscr{D}\,\delta[\varphi]\\
\mathscr{D}^{\frac{1}{2}}\mathscr{D}^{\frac{1}{2}}\,T_{H}[\varphi]&=\mathscr{D}\,T_{H}[\varphi]
\end{split}
\end{equation*}
\subsubsection{Polynomial functions}
Let $T_{f}\left[\varphi\right]$ be the regular distribution which is generated by the polynomially growing function $f(x)=x^{n}$ where $n$ is an arbitray positive integer $n\in \mathbb{N}$.\\
Then: 
\begin{rem}
\begin{equation*}
\mathscr{D}^{\alpha}T_{f}[\varphi]=\begin{cases}
0& \mathrm{if} \alpha>n\\
\frac{\Gamma(n+1)}{\Gamma(n-\alpha+1)}x^{n-\alpha}\,T_{\bold{1}}[\varphi]&\mathrm{if} \alpha\leq n\\
\end{cases}
\end{equation*}
\end{rem}
\begin{proof}
We start by a Fourier transform:
\begin{equation}
\label{dirac_frac}
\begin{split}
\mathscr{F}\mathscr{D}^{\alpha}T_{f}&=i^{\alpha}x^{\alpha}\mathscr{F}T_f[\varphi]=2\pi x^{\alpha} i^{n+\alpha}\mathscr{D}^{n}\delta[\varphi]\\
&=2\pi\,i^{n+\alpha}(-1)^{n}n!\frac{\delta[\varphi]}{x^{n-\alpha}}
\end{split}
\end{equation}
The latter expression does vanish for every real $\alpha$ larger than $n$. This in turn means that $\mathscr{D}^{\alpha}T_{f}=0$. Therefore, one is left with the case $\alpha\leq n$:\\
By presuming that the formula: 
\begin{equation*}
\mathscr{D}^{n}\,\delta\left[\varphi\right]=(-1)^{n}\,n!\,\frac{\delta[\varphi]}{x^{n}},\quad n\in\mathbb{N}
\end{equation*}
holds also for fractional powers of the operator $\mathscr{D}$ one may write:
\begin{equation*}
\mathscr{D}^{n-\alpha}\,\delta\left[\varphi\right]=(-1)^{n-\alpha}\,\Gamma(n-\alpha+1)\frac{\delta[\varphi]}{x^{n-\alpha}}
\end{equation*}
Thus, the right-hand side of Eq.~\ref{dirac_frac} may be rewritten in the form:
\begin{equation}
\label{introduction_frac_delta}
\mathscr{F}\mathscr{D}^{\alpha}T_{f}=2\pi\,i^{n+\alpha}(-1)^{\alpha}\frac{\Gamma(n+1)}{\Gamma(n-\alpha+1)}\mathscr{D}^{n-\alpha}\,\delta\left[\varphi\right]
\end{equation}
By introducing the following identity into Eq.~\ref{introduction_frac_delta}:
\begin{equation*}
\mathscr{F}^{-1}\left[\mathscr{D}^{n-\alpha}\,\delta\right]=\frac{1}{2\pi}\,(-ix)^{n-\alpha}\,T_{\bold{1}}[\varphi]
\end{equation*}
we obtain a general formula valid for every real number $\alpha$:
\begin{equation}
\label{prefac}
\mathscr{D}^{\alpha}T_{f}[\varphi]=\underbrace{i^{n+\alpha}(-i)^{n-\alpha}(-1)^{\alpha}}_{\bigstar}\begin{cases}
0& \mathrm{if} \alpha>n\\
\frac{\Gamma(n+1)}{\Gamma(n-\alpha+1)}x^{n-\alpha}\,T_{\bold{1}}[\varphi]&\mathrm{if} \alpha\leq n\\
\end{cases}
\end{equation}
The prefactor reduces to unity:
\begin{equation}
\label{prefac=1}
\bigstar=\mathrm{e}^{(n+\alpha)i\frac{\pi}{2}}\,\mathrm{e}^{-i(n-\alpha)\frac{\pi}{2}}\,\mathrm{e}^{-i\pi\alpha}=\mathrm{e}^{0}=1,\quad \forall n\in \mathbb{N},\,\alpha \in \mathbb{R}
\end{equation}
which is not the case in every branch of the complex logarithm.
\end{proof}
In the particular case of a vanishing exponent $n=0$ one obtains:
\begin{equation*}
\frac{\mathrm{d}^{\alpha}}{\mathrm{d}x^{\alpha}}1=\begin{cases}
0& \mathrm{if} \alpha>0\\
\frac{1}{\Gamma(1-\alpha)}x^{-\alpha}&\mathrm{if} \alpha\leq 0\\
\end{cases}
\end{equation*}
We note also that for every real $\alpha$ smaller or equal than any (fixed) integer $n$ formula \ref{prefac} is analogous to the classical formula:
\begin{equation}
\label{classical_formula}
\frac{\mathrm{d}^{\alpha}}{\mathrm{d}x^{\alpha}}\,x^{n}=\frac{n!}{(n-\alpha)!}\,x^{n-\alpha}=\frac{\Gamma(n+1)}{\Gamma(n-\alpha+1)}\,x^{n-\alpha}
\end{equation}
However, in contrast to Eq.~\ref{prefac} the classical formula \ref{classical_formula} remains undefined for every integer $\alpha$ larger than $n$.
Moreover, it does not vanish in cases such as:
\begin{equation*}
\mathrm{d}^{\frac{3}{2}}/\mathrm{d}x^{\frac{3}{2}}x\neq 0
\end{equation*}
which is an undesired feature.\\
\begin{exm}
According to Eq.~\ref{prefac} one obtains the formula:
\begin{equation}
\label{hospital_example}
\mathscr{D}^{\frac{1}{2}}T_{x}[\varphi]=\frac{\Gamma(2)}{\Gamma(\frac{3}{2})}\sqrt{x}\,T_{\bold{1}}[\varphi]=\frac{2}{\sqrt{\pi}}T_{\sqrt{x}}\,[\varphi]
\end{equation}
This example is mainly interesting from a historical perspective. The issue of fractional derivatives has been raised in a series of correspondences between Guillaume de L' H\^ospital and Leibniz at the end of the 17th century. In one of these exchanges \cite{leibniz:1695} Leibniz raises the question what $\mathrm{d}^{\frac{1}{2}}/\mathrm{d}x^{\frac{1}{2}}x$ is. Formula \ref{hospital_example} is an adequate answer to this puzzle.
It could be written in the following sloppily way:
\begin{equation*}
\frac{\mathrm{d}^{\frac{1}{2}}}{\mathrm{d}x^{\frac{1}{2}}}\,x=2\sqrt{\frac{x}{\pi}}
\end{equation*}
\end{exm}
\subsubsection{Rational functions}
\label{rational_functions}
The purpose of the present section is to evaluate the fractional derivative $\mathscr{D}^{\alpha}$ of the tempered distribution
$\mathrm{Pf}\,\frac{1}{x^n}$ with $n$ being a positive integer. First, we consider the case where $\alpha$ is a real but not integer number:
{$\alpha\in \mathbb{R}\setminus{\mathbb{Z}}$}:\\
As usual, we begin with the following Fourier transform:
\begin{equation*}
\mathscr{F}\mathscr{D}^{\alpha}\,\mathrm{Pf}\frac{1}{x^n}=i^{\alpha}x^{\alpha}\mathscr{F}\mathrm{Pf}\,\frac{1}{x^n}=i^{\alpha}x^{\alpha}T_g[\varphi]\\
\end{equation*}
where (according to Eq.~\ref{formula_power_negative}):
\begin{equation*}
g(x)=\pi\,\frac{x^{n-1}}{i^{n}(n-1)!}\,\mathrm{sgn}{(x)}
\end{equation*}
By means of Def.~\ref{firstdefinition} one may write:
\begin{equation*}
i^{\alpha}x^{\alpha}T_g[\varphi]=T_{h^{\star}}[\varphi]
\end{equation*}
with:
\begin{equation}
\label{replace}
h^{\star}(x)=i^{\alpha}x^{\alpha}g(x)=\frac{\pi \, x^{n+\alpha-1}}{i^{n-\alpha}\,\Gamma(n)}\text{sgn}{(x)}
\end{equation}
In order to determine the Fourier inverse transform of $T_{h^{\star}}[\varphi]$ we replace  in Eq.~\ref{replace} the signum function $\mathrm{sgn}(x)$ by $2H(x)-1$ and use the following identities afterwards:
\begin{equation*}
\begin{split}
&\mathscr{F}^{-1}\,\mathrm{Pf}\,x^{n+\alpha-1}=\frac{\Gamma(n+\alpha)}{2\pi}\,\left[(-1)^{n+1}\, \mathrm{e}^{-i\pi\alpha}\,\mathrm{Pf}(ix)^{-n-\alpha}+\mathrm{Pf}(-ix)^{-n-\alpha}\right]\\
&\mathscr{F}^{-1}\,\mathrm{Pf}\left(Hx^{n+\alpha-1}\right)=\frac{\Gamma(n+\alpha)}{2\pi}\,\mathrm{Pf}(-ix)^{-n-\alpha}
\end{split}
\end{equation*}
Finally, this gives:
\begin{equation*}
\mathscr{D}^{\alpha}\,\mathrm{Pf}\frac{1}{x^n}=i^{\alpha-n}\frac{\Gamma(n+\alpha)}{\Gamma(n)}\left[\mathrm{Pf}(-ix)^{-n-\alpha}-(-1)^{n+1}\,\mathrm{e}^{-i\pi\alpha}\,\mathrm{Pf}(ix)^{-n-\alpha}\right]\\
\end{equation*}
\begin{exm}
\begin{equation*}
\begin{split}
\mathscr{D}^{+\frac{1}{2}}\mathrm{Vp}\frac{1}{x}&=-\mathrm{Pf}\,\frac{H(-x)}{|x|^{\frac{3}{2}}}
\\
\mathscr{D}^{-\frac{1}{2}}\mathrm{Vp}\frac{1}{x}&=-2\pi\mathrm{Pf}\,\frac{H(-x)}{\sqrt{|x|}}
\end{split}
\end{equation*}
\end{exm}
Next, we evaluate the case where $\alpha$ is equal to an integer $m$ fulfilling the inequality 
$n-1+m\geq 0$.
We want to evaluate the following derivative:
\begin{equation*}
\mathscr{D}^{m}\,\mathrm{Pf}\,x^{-n}
\end{equation*}
Fourier transforming the latter expression yields:
\begin{equation*}
\mathscr{F}\mathscr{D}^{m}\,\mathrm{Pf}\,x^{-n}=i^mx^{m}\mathscr{F}\,\mathrm{Pf}\,x^{-n}=T_{g^{\star}}[\varphi]
\end{equation*}
where:
\begin{equation*}
g^{\star}(x)=\frac{\pi \, x^{n+m-1}}{i^{n-m}\,\Gamma(n)}\,\mathrm{sgn}{(x)}
\end{equation*}
Fourier inversion $\mathscr{F}^{-1}$ leads to:
\begin{equation}
\label{the_formula}
\mathscr{D}^{m}\,\mathrm{Pf}\frac{1}{x^n}=(-1)^{m}\,\frac{\Gamma(n+m)}{\Gamma(n)}\,\mathrm{Pf}\frac{1}{x^{n+m}}
\end{equation}
\begin{exm}
In particular, for $n=m=1$ formula \eqref{the_formula} reduces to the well-known case:
\begin{equation*}
\mathscr{D}\,\mathrm{Vp}\frac{1}{x}=-\mathrm{Pf}\frac{1}{x^2}
\end{equation*}
\end{exm}
Next, we consider the complementary case where $n-1+m<0$. Our purpose is to evaluate the expression:
\begin{equation}
\label{challenge}
\mathscr{D}^{m}\,\mathrm{Pf}\,x^{-n}
\end{equation}
Fourier transforming Eq.~\ref{challenge} gives (using Theorem \ref{theorem4}):
\begin{equation}
\label{equation_1}
\mathscr{F}\mathscr{D}^{m}\,\mathrm{Pf}\,x^{-n}=i^{m}x^{m}\mathscr{F}\mathrm{Pf}\,x^{-n}=\pi\frac{i^{-n+m}}{(n-1)!}\,x^{n-1+m}\,\mathrm{sgn}(x)
\end{equation}
In order to determine the inverse Fourier transform $\mathscr{F}^{-1}$ on the right-hand side of Eq.~\ref{equation_1} we use the following Lemma:
\begin{lem}
\begin{equation}
\label{equation_2}
\mathscr{D}^{n}\mathrm{Pf}\frac{1}{|x|}=(-1)^{n}n!\,\mathrm{Pf}\,\frac{\mathrm{sgn}(x)}{x^{n+1}}-2\,\mathscr{D}^{n}\delta[\varphi]\sum^{n}_{i=1}\frac{1}{i}
\end{equation}
\end{lem}
\begin{proof}
First, we write:
\begin{equation}
\mathrm{Pf}\frac{1}{|x|}=\mathrm{Pf}\,\frac{\mathrm{sgn}(x)}{x}=\left(\mathrm{Pf}\frac{1}{x}\right)\,\mathrm{sgn}(x)
\end{equation}
According to the general Leibniz rule for the nth derivative of a product  we have the following expression:
\begin{equation}
\mathscr{D}^{n}\left[\left(\mathrm{Pf}\frac{1}{x}\right)\,\mathrm{sgn}(x)\right]=\sum^{n}_{i=1}
\binom{n}{i}\mathscr{D}^{n-i}\,\left(\mathrm{Pf}\frac{\mathrm{1}}{x}\right)\,[\mathrm{sgn}(x)]^{(i)}
\end{equation}
Since 
\begin{equation*}
\begin{split}
&\mathscr{D}^{n-i}\,\left(\mathrm{Pf}\frac{\mathrm{1}}{x}\right)=(-1)^{n-i}(n-i)!\,\mathrm{Pf}\frac{\mathrm{1}}{x^{n-i+1}}\\
&[\mathrm{sgn}(x)]^{(i)}=2\,\delta^{(i-1)}[\varphi]=2\,(-1)^{i-1}(i-1)!\,\frac{\delta[\varphi]}{x^{i-1}},\quad i=1,2,3\dots
\end{split}
\end{equation*}
we get then:
\begin{equation*}
\begin{split}
\mathscr{D}^{n}\left[\left(\mathrm{Pf}\frac{1}{x}\right)\,\mathrm{sgn}(x)\right]&=
(-1)^{n}n!\,\mathrm{Pf}\,\frac{\mathrm{sgn}(x)}{x^{n+1}}+
2
\sum^{n}_{i=1}\frac{1}{i}(-1)^{n-1}\,\frac{\delta[\varphi]}{x^{n}}\\
&=(-1)^{n}n!\,\mathrm{Pf}\,\frac{\mathrm{sgn}(x)}{x^{n+1}}\,-2\,\mathscr{D}^{n}\delta[\varphi]\sum^{n}_{i=1}\frac{1}{i}
\end{split}
\end{equation*}
\end{proof}
Next, we apply a Fourier transform on both sides of Eq.~\ref{equation_2} and use the following formula \cite{gelfand:1964}:
\begin{equation}
\mathscr{F}\mathrm{Pf}\frac{1}{|x|}=-2\gamma-2\log{|x|}
\end{equation}
where $\gamma=0.5772156\dots$ denotes the Euler-Mascheroni constant.\\
\\
We immediately obtain:
\begin{equation}
\label{beautiful}
\mathscr{F}\mathrm{Pf}\,\frac{\mathrm{sgn}(x)}{x^{n+1}}=2\,\frac{(-ix)^{n}}{n!}\left[\sum^{n}_{i=1}\frac{1}{i}-\gamma-\log{|x|}\right]
\end{equation}
Now, we apply formula \ref{beautiful} to \ref{equation_1} by replacing the exponent $n$ with $-n-m$. Then we obtain finally the interesting formula:
\begin{equation}
\label{final_formula}
\mathscr{D}^{m}\,\mathrm{Pf}\,x^{-n}=\frac{(-1)^{n}}{(n-1)!\,(-n-m)!}\left[\sum^{-n-m}_{i=1}\frac{1}{i}-\gamma-\log{|x|}\right]x^{-n-m}
\end{equation}
\begin{rem}
Formula \ref{final_formula} can be compared with an analogous formula for $\mathscr{D}^{m}\,\mathrm{Pf}\,|x|^{-n}$ which is easily derived by using the Fourier transform of $\mathrm{Pf}\,|x|^{-n}$\cite{gelfand:1964}. The formula is given by:
\begin{equation}
\label{formula_gelfand}
\mathscr{D}^{m}\,\mathrm{Pf}\,|x|^{-n}=A
\begin{cases}
i\,d_{0}^{\,2k}x^{2k-1}-i\,d_{-1}^{\,2k}\,x^{2k-1}\ln{|x|}&\mathrm{if}\quad 1-n-m=2k\\
&\,\\
c_{0}^{(2k+1)}x^{2k}-c_{-1}^{(2k+1)}x^{2k}\ln{|x|}&\mathrm{if}\quad 1-n-m=2k+1
\end{cases}
\end{equation}
With
\begin{equation*}
\begin{split}
&a_{0}^{(l)}=\frac{i^{l-1}}{(l-1)!}\left[1+\frac{1}{2}+\cdots+\frac{1}{l-1}+\underbrace{\Gamma^{\prime}(1)}_{=-\gamma}+i\frac{\pi}{2}\right]\\
&A=\frac{i^{-n+m}}{2(n-1)!}
\end{split}
\end{equation*}
and
\begin{equation*}
\begin{split}
c_{0}^{(2k+1)}&=2\,\operatorname{Re}{a_{0}^{2k+1}},\quad c_{-1}^{(2k+1)}=2\frac{(-1)^{k}}{(2k)!}\\
d_{0}^{(2k)}&=2\,\operatorname{Im}{a_{0}^{2k}},\quad d_{-1}^{(2k)}=2\frac{(-1)^{k+1}}{(2k-1)!}
\end{split}
\end{equation*}
A short calculation shows that formula \ref{final_formula} coincide with \ref{formula_gelfand} for even $n$.
\end{rem}
\begin{exm}
\begin{equation*}
\begin{split}
\mathscr{D}^{-1}\,\frac{1}{x}&=\gamma+\log{|x|}\\
\mathscr{D}^{-2}\,\frac{1}{x}&=(\gamma+\log{|x|}-1)x
\end{split}
\end{equation*}
\end{exm}
\section{fractional derivatives of periodic functions}
The Schwartz-space of smooth and rapidly decreasing functions $\mathscr{S(\mathbb{R})}$ is a very restricted functional space, its dual $\mathscr{S^{\prime}(\mathbb{R})}$ being therefore a "large" space in the sense that it includes \eg the spaces $L^{p}(\mathbb{R})$, $1\leq p \leq\infty$.    
The operator $\mathcal{D}^{\alpha}$ acts on tempered distributions $T\in\mathscr{S^{\prime}(\mathbb{R})}$.  
It seems impossible to introduce an operator acting on \emph{functions} (in the classical sense) which preserves the properties i)\,-\,iv) of Sec.~\ref{intro}.
However, there is a restricted class of functions for which a consistent definition exists as well. As already mentioned in Sec.~\ref{sec_frac_deriv} a consistent definition for periodic functions (denoted by $\frac{\mathrm{d}^{\alpha}}{\mathrm{d}x^{\alpha}}f(x)$) exists via Eq.~\ref{half_derivative}.\\
\\
Without loss of generality, we assume $f(x)$ to be a $2\pi$-periodic function with the Fourier series:
\begin{equation*}
f(x)=\sum_{n=-\infty}^{\infty}c_{n}\mathrm{e}^{inx}=c_{0}+\sideset{}{'}\sum^{\infty}_{n=-\infty}c_{n}\mathrm{e}^{inx},\quad c_{n}\in \mathbb{C}
\end{equation*}
According to Eq.~\ref{half_derivative} and by regularizing the sum by the \emph{Abel summation} technique ($\mathcal{A}$) we define:
\begin{equation*}
\frac{\mathrm{d}^{\alpha}}{\mathrm{d}x^{\alpha}}f(x):=c_{0}\,\frac{\mathrm{d}^{\alpha}}{\mathrm{d}x^{\alpha}}1+\mathcal{A}-\sum_{n=-\infty}^{\infty}c_{n}(in)^{\alpha}\mathrm{e}^{inx},\quad \alpha \in \mathbb{R}
\end{equation*}
where by Eq.~\ref{prefac}:
\begin{equation*}
\frac{\mathrm{d}^{\alpha}}{\mathrm{d}x^{\alpha}}1=\begin{cases}
0& \mathrm{if} \alpha>0\\
\frac{1}{\Gamma(1-\alpha)}x^{-\alpha}&\mathrm{if} \alpha\leq 0\\
\end{cases}
\end{equation*}
and Abel sum:
\begin{equation*}
\mathcal{A}-\sum_{n=-\infty}^{\infty}c_{n}(in)^{\alpha}\mathrm{e}^{inx}:=\lim_{\epsilon\to0^{+}}\sum_{n=-\infty}^{\infty}c_{n}(in)^{\alpha}\mathrm{e}^{inx-\epsilon|n|}
\end{equation*}
\begin{exm}
The first example we consider is given by the following function:
\begin{equation*}
f(x)=x \quad \text{for}-\pi<x<\pi \quad \mathrm{(periodically\,continued)}
\end{equation*}
Where its Fourier series is given by:
\begin{equation*}
f(x)=2\sum_{n=1}^{\infty}(-1)^{n-1}\frac{\sin nx}{n}
\end{equation*}
According to Eq.~\ref{half_deriv_sin_cos} one obtains:
\begin{equation*}
\frac{\mathrm{d}^{\alpha}}{\mathrm{d}x^{\alpha}}f(x)=i^{\alpha-1}\mathcal{A}-\sum_{n=1}^{\infty}(-1)^{n-1}n^{\alpha-1}\left[\mathrm{e}^{inx}-\mathrm{e}^{-inx-i\pi \alpha}\right]
\end{equation*}
In general, the fractional derivative of a real valued periodic function turns into a complex-valued one, except for some specific values such as \eg $\alpha=\pm\frac{1}{2}$:
\begin{equation}
\begin{split}
\label{weierstrass}
&\frac{\mathrm{d}^{+\frac{1}{2}}}{\mathrm{d}x^{+\frac{1}{2}}}f(x)=\sqrt{2}\sum_{n=1}^{\infty}(-1)^{n-1}\frac{\sin{nx}+\cos{nx}}{\sqrt{n}}\\
&\frac{\mathrm{d}^{-\frac{1}{2}}}{\mathrm{d}x^{-\frac{1}{2}}}f(x)=\sqrt{2}\sum_{n=1}^{\infty}(-1)^{n-1}\frac{\sin{nx}-\cos{nx}}{n^{\frac{3}{2}}}
\end{split}
\end{equation}
\end{exm}
\begin{exm}
The second example we consider is given by the function:
\begin{equation*}
g(x)=|x| ,\quad\mathrm{for}-\pi<x<\pi \quad \mathrm{(periodically\,continued)}
\end{equation*}
with the Fourier series:
\begin{equation*}
g(x)=\frac{\pi}{2}-\frac{4}{\pi}\sum_{n=1}^{\infty}\frac{\cos{(2n+1)x}}{(2n+1)^2}
\end{equation*}
Again, by Eq.~\ref{half_deriv_sin_cos}
\begin{equation*}
\frac{\mathrm{d}^{\alpha}}{\mathrm{d}x^{\alpha}}f(x)=-\frac{2}{\pi}i^{\alpha}\mathcal{A}-\sum_{n=1}^{\infty}(2n+1)^{\alpha-2}\left[\mathrm{e}^{i(2n+1)x}+\mathrm{e}^{-i(2n+1)x-i\pi \alpha}\right]
\end{equation*}
Specifically, for $\alpha=\pm\frac{1}{2}$:
\begin{equation}
\begin{split}
\label{2nd_example}
&\frac{\mathrm{d}^{+\frac{1}{2}}}{\mathrm{d}x^{+\frac{1}{2}}}g(x)=-\frac{4}{\pi\sqrt{2}}\sum_{n=1}^{\infty}\frac{\cos{(2n+1)x}-\sin{(2n+1)x}}{\sqrt{2n+1}(2n+1)}\\
&\frac{\mathrm{d}^{-\frac{1}{2}}}{\mathrm{d}x^{-\frac{1}{2}}}g(x)=\sqrt{\pi x}-\frac{4}{\pi\sqrt{2}}\sum_{n=1}^{\infty}\frac{\cos{(2n+1)x}+\sin{(2n+1)x}}{\sqrt{2n+1}(2n+1)^2}
\end{split}
\end{equation}
\\
\end{exm}
\begin{figure}[h!]
  \centering
  \subfloat[][]{\includegraphics[width=0.4\linewidth]{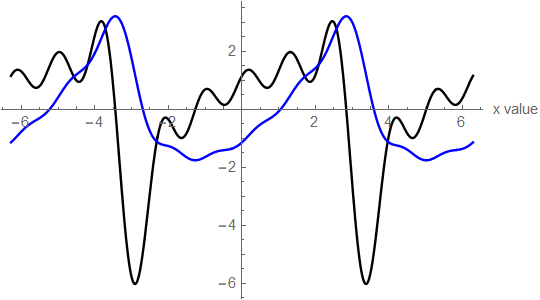}}
  \qquad
  \subfloat[][]{\includegraphics[width=0.4\linewidth]{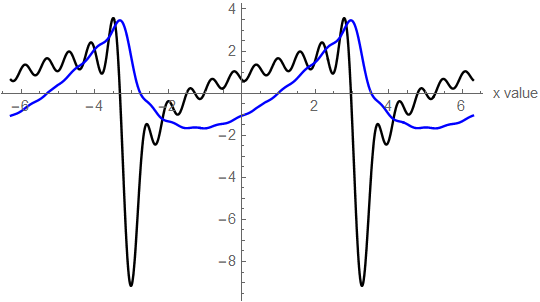}}
 \qquad
  \subfloat[][]{\includegraphics[width=0.4\linewidth]{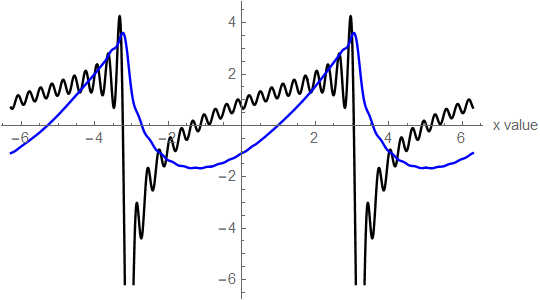}}
 \qquad
  \subfloat[][]{\includegraphics[width=0.4\linewidth]{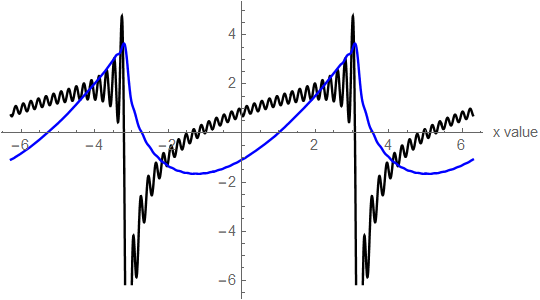}}
  \caption{Fourier series of $\frac{\mathrm{d}^{+\frac{1}{2}}}{\mathrm{d}x^{+\frac{1}{2}}}f(x)$ (black) and $\frac{\mathrm{d}^{-\frac{1}{2}}}{\mathrm{d}x^{-\frac{1}{2}}}f(x)$ (blue) up to order: $n=5$ (a), $n=10$ (b), $n=20$ (c), $n=30$ (d).}
\label{Bild}
\end{figure}
As shown in Fig.~\ref{Bild}, the graph of the truncated Fourier series of $\mathrm{d}^{\frac{1}{2}}/\mathrm{d}x^{\frac{1}{2}}\,f(x)$ is oscillating more and more rapidly with increasing order $n$. One can therefore conjecture that it is converging to a function of the Weierstrass type with the property of being continous but nowhere differentiable \cite{weierstrass:1895}.
\newpage
\appendix
\section{The Principal Value distribution $\bold{\mathrm{Vp}\,\frac{1}{x}}$}
\label{app1}
\subsection{Derivative of $\bold{\mathrm{Vp}\,\frac{1}{x}}$}
The derivative of the distribution $\mathrm{Vp}\,\frac{1}{x}$ is formally given by:
\begin{equation*}
\mathscr{D}\,\mathrm{Vp}\frac{1}{x}[\varphi]:=-\mathrm{Vp}\frac{1}{x}[\varphi^{\prime}]=-\lim_{\epsilon\to0^{+}}\int\limits_{|x|>\epsilon}\frac{\varphi^{\prime}(x)}{x}\mathrm{d}x
\end{equation*}
Integration by parts gives:
\begin{equation}
\begin{split}
\mathscr{D}\,\mathrm{Vp}\frac{1}{x}&=-\lim_{\epsilon\to0^{+}}\int\limits_{|x|>\epsilon}\frac{\varphi(x)}{x^2}\mathrm{d}x-\frac{\varphi(\epsilon)+\varphi(-\epsilon)}{\epsilon}\\
&=-\lim_{\epsilon\to0^{+}}\int\limits_{|x|>\epsilon}\frac{\varphi(x)}{x^2}\mathrm{d}x-\frac{2\,\varphi(0)}{\epsilon}=-\mathrm{Pf}\frac{1}{x^2}
\end{split}
\end{equation}
\subsection{Fourier transform of $\bold{\mathrm{Vp}\,\frac{1}{x}}$}
We start by evaluating the following Fourier transform:
\begin{equation*}
\begin{split}
\mathscr{F}\mathrm{Vp}\,\frac{\mathrm{e}^{-\epsilon|x|}}{x}
:=\lim_{\epsilon\to 0^{+}}\int\limits_{|x|>\epsilon}\frac{\hat{\varphi}(x)}{x}\mathrm{e}^{-\epsilon|x|}\mathrm{d}x
\end{split}
\end{equation*}
By Fubini's Theorem we obtain:
\begin{equation}
\label{reference_fub}
\lim_{\epsilon\to 0^{+}}\int\limits_{\mathbb{R}}\mathrm{d}k\,\varphi(k)\int\limits_{|x|>\epsilon} \frac{\mathrm{e}^{-\epsilon|x|-ikx}}{x}\mathrm{d}x
\end{equation}
Which is equal to:
\begin{equation}
\label{limes}
-2i\lim_{\epsilon\to 0^{+}}\int\limits_{\mathbb{R}}\mathrm{d}k\,\varphi(k)\int^{\infty}_{\epsilon}\frac{\mathrm{e}^{-\epsilon|x|}}{x}\sin{kx}\,\mathrm{d}x
\end{equation}
The limit $\epsilon\rightarrow 0^{+}$ in \ref{limes} can be evaluated by using the fact that
the Fourier transform $\mathscr{F}$ is continous in $\mathcal{S}^{\prime}(\mathbb{R})$:
\begin{equation*}
\lim_{\epsilon\to0^{+}}\mathscr{F}\mathrm{Vp}\,\frac{\mathrm{e}^{-\epsilon|x|}}{x}=
\mathscr{F}\mathrm{Vp}\,\frac{1}{x}=-2i\int\limits_{\mathbb{R}}\mathrm{d}k\,\varphi(k)\int^{\infty}_{0}\frac{\sin{kx}}{x}\,\mathrm{d}x
\end{equation*}
The latter integral (Dirichlet) is given by:
\begin{equation*}
\int^{\infty}_{0}\frac{\sin{kx}}{x}\,\mathrm{d}x=\frac{\pi}{2}\,\mathrm{sgn}(k)
\end{equation*}
Therefore:
\begin{equation*}
\mathscr{F}\mathrm{Vp}\frac{1}{x}[\varphi]=-i\pi\int\limits_{\mathbb{R}}\mathrm{d}k\,\varphi(k)\,\mathrm{sgn}(k) 
\end{equation*}
which proves the formula:
\begin{equation*}
\mathscr{F}\mathrm{Vp}\frac{1}{x}=-i\pi\,\mathrm{sgn}[\varphi]
\end{equation*}
\section{derivative of the Logarithm}
\label{app2}
Let $f(x)=\log{x}$ and $g(x)=\log{|x|}$ be the functions generating the regular distributions:
\begin{equation*}
\begin{split}  
T_{f}[\varphi]:&=\mathrm{Vp}\int\limits_{\mathbb{R}}\log{x}\,\varphi(x)\mathrm{d}x=\lim_{\epsilon\to 0^{+}}\int\limits_{\mathbb{R}\setminus[-\epsilon,\epsilon]}\log{x}\,\varphi(x)\mathrm{d}x\\
T_{g}[\varphi]:&=\mathrm{Vp}\int\limits_{\mathbb{R}}\log{|x|}\,\varphi(x)\mathrm{d}x=\lim_{\epsilon\to 0^{+}}\int\limits_{\mathbb{R}\setminus[-\epsilon,\epsilon]}\log{|x|}\,\varphi(x)\mathrm{d}x
\end{split}
\end{equation*}
Then 
\begin{equation}
\label{log_dist}
\mathscr{D}\,T_{f}[\varphi]=\mathrm{Vp}\frac{1}{x}+i\pi\,\delta[\varphi]
\end{equation}
\begin{equation}
\label{log_dist2}
\mathscr{D}\,T_{g}[\varphi]=\mathrm{Vp}\frac{1}{x}
\end{equation}
\begin{proof}
\begin{equation*}
\begin{split}
\mathscr{D}\,T_{f}[\varphi]&=-T_{f}\,[\varphi^{\prime}]=-\lim_{\epsilon\to 0^{+}}\int\limits_{\mathbb{R}\setminus[-\epsilon,\epsilon]}\log{x}\,\varphi^{\prime}(x)\mathrm{d}x\\
&=-\lim_{\epsilon\to 0^{+}}\,\int\limits_{\epsilon}^{\infty}\left[\log{(-x)}\,\varphi^{\prime}(-x)
+\log{x}\,\varphi^{\prime}(x)\right]\mathrm{d}x\,
\end{split}
\end{equation*}
By inserting the identity $\log{(-x)}=\log{(|x|)}-i\pi H(x)$ and performing partial integrations one obtains:
\begin{equation*}
\begin{split}
\mathscr{D}\,T_{f}[\varphi]&=\lim_{\epsilon\to 0^{+}}i\pi\varphi(-\epsilon)+\log{\epsilon}\underbrace{\left[\varphi(\epsilon)-\varphi(-\epsilon)\right]}_{\mathcal{O}(\epsilon)}\\
&+\lim_{\epsilon\to 0^{+}}\int\limits_{\epsilon}^{\infty}\frac{\varphi(x)-\varphi(-x)}{x}\mathrm{d}x=\mathrm{Vp}\frac{1}{x}+i\pi\,\delta[\varphi]
\end{split}
\end{equation*}
Proof of Eq.~\ref{log_dist2} is straightforward and well known \cite{gelfand:1964}. By taking the derivative on both sides of the identity (For convenience we write: $\mathscr{D}\,T_{f}\equiv \frac{d}{dx}\log{x}$)
\begin{equation*}
\log{x}=\log{|x|}-i\pi\,H(-x)
\end{equation*}
we obtain:
\begin{equation*}
\frac{d}{dx}\log{x}\underline{=}_{\eqref{log_dist}}\mathrm{Vp}\frac{1}{x}+i\pi\,\delta[\varphi]=\frac{d}{dx}\log{|x|}-i\pi\,\underbrace{\frac{d}{dx}H(-x)}_{=-i\pi\,\delta}
\end{equation*}
From which follows Eq.~\eqref{log_dist2}.
\end{proof}
\begin{rem}
The identity \ref{log_dist} was formulated by Dirac \cite{dirac:1947} well before the invention of distribution theory in the form:
\begin{equation}
\label{dirac_logarithm}
\frac{d}{dx}\log{x}=\frac{1}{x}+i\pi\,\delta(x)
\end{equation}
To be precise, in \cite{dirac:1947} Dirac wrote equation \ref{dirac_logarithm} with a minus sign, the reason being that he uses a logarithmic branch where $\log{(-1)}=+i\pi$ is valid.
\end{rem}
\section{Fourier transform of the Logarithm}
The Fourier transform of the logarithm is readily evaluated by starting with the identity \cite{gelfand:1964}
\begin{equation}
\label{fundamental}
\mathscr{F}\mathrm{Pf}\frac{1}{|x|}=-2\gamma-2\log{|x|}
\end{equation}
Taking the Fourier transform on both sides of \eqref{fundamental} leads to
\begin{equation}
\mathscr{F}\log{|x|}=-\pi\,\mathrm{Pf}\frac{1}{|x|}-2\pi\gamma\,\delta[\varphi]
\end{equation}
Next, we use:
\begin{equation}
\label{logarithm}
\log{x}=\log{|x|}-i\pi\,H(-x)
\end{equation}
Fourier transforming both sides of Eq.~\eqref{logarithm} yields:
\begin{equation*}
\mathscr{F}\log{x}=-\pi\,\mathrm{Pf}\frac{1}{|x|}-2\pi\gamma\,\delta[\varphi]-i\pi\,\mathscr{F}H(-x)
\end{equation*}
By using:
\begin{equation*}
\mathscr{F}H(-x)=\pi\,\delta[\varphi]+i\,\mathrm{Vp}\frac{1}{x}
\end{equation*}
we get:
\begin{equation*}
\mathscr{F}\log{x}=-\pi\,\mathrm{Pf}\frac{1}{|x|}+\pi\mathrm{Vp}\frac{1}{x}-(2\pi\gamma+i\pi^2)\,\delta[\varphi]
\end{equation*}
which is equal to:
\begin{equation*}
\mathscr{F}\log{x}=-2\pi\,\mathrm{Pf}\frac{H(-x)}{|x|}-(2\pi\gamma+i\pi^2)\,\delta[\varphi]
\end{equation*}
\appendix*
\newpage
\bibliography{References_paper_fractional_derivatives}
\bibliographystyle{plain}
\end{document}